\documentclass[a4paper,10pt,reqno, english]{amsart}

\usepackage[utf8]{inputenc}
\usepackage{amssymb}
\usepackage{graphicx}
\usepackage{amsmath,bm,amsthm}
\usepackage{amsfonts,amssymb,enumerate}

\usepackage{anysize}

\usepackage{url,paralist}
\usepackage{mathtools}
\usepackage{float}
\usepackage{dsfont}
\usepackage{caption}
\usepackage{subcaption}
 \usepackage{makecell}
\usepackage[arrow,curve,matrix,tips,2cell]{xy}
  \SelectTips{eu}{10} \UseTips
  \UseAllTwocells

\usepackage{hyperref}
\hypersetup{
  colorlinks   = true, 
  urlcolor     = blue, 
  linkcolor    = blue, 
  citecolor   = red 
}
\usepackage{color}
\usepackage{enumerate,amssymb}

\theoremstyle{plain}
\newtheorem{theorem}{Theorem}[section]

\newtheorem{corollary}[theorem]{Corollary}
\newtheorem{proposition}[theorem]{Proposition}
\newtheorem{conjecture}[theorem]{Conjecture}
\newtheorem*{theorem*}{Theorem}

\newtheorem*{r-conjecture}{The Ramos conjecture}

\newtheorem*{claim*}{Claim}

\theoremstyle{definition}

\newtheorem{example}[theorem]{Example}

\newtheorem{remark}[theorem]{Remark}

\newcommand{\R}{\mathbb{R}}

\newcommand{\C}{\mathcal{C}}
\newcommand{\N}{\mathbb{N}}
\newcommand{\Z}{\mathbb{Z}}

\newcommand{\HH}{\mathcal{H}}
\newcommand{\bb}{\mathfrak h}

\newcommand\Sym{\mathfrak{S}}
\newcommand\oo{\mathfrak{o}}
\newcommand\Wk{\Sym^\pm}

\newcommand{\im}{\operatorname{im}}
\newcommand{\conn}{\operatorname{conn}}

\newcommand{\sk}{\operatorname{sk}}
\newcommand{\relint}{\operatorname{relint}}

\allowdisplaybreaks

\begin{document}

\title{More bisections by hyperplane arrangements}

\author[Blagojevi\'c]{Pavle V. M. Blagojevi\'{c}}
\thanks{The research by Pavle V. M. Blagojevi\'{c} leading to these results has
        received funding from the grant ON 174024 of the Serbian Ministry of Education and Science, and from DFG via the Collaborative Research Center TRR 109 ``Discretization in Geometry and Dynamics''.}
\address{Inst. Math., FU Berlin, Arnimallee 2, 14195 Berlin, Germany\hfill\break
\mbox{\hspace{4mm}}Mat. Institut SANU, Knez Mihailova 36, 11001 Beograd, Serbia}
\email{blagojevic@math.fu-berlin.de}
\email{pavleb@mi.sanu.ac.rs}
\author[Dimitrijevi\'{c} Blagojevi\'{c}]{Aleksandra Dimitrijevi\'{c} Blagojevi\'{c}}
\thanks{The research by Aleksandra Dimitrijevi\'{c} Blagojevi\'{c} leading to these results has
        received funding from the grant ON 174024 of the Serbian Ministry of Education and Science.}
\address{Mat. Institut SANU, Knez Mihailova 36, 11001 Beograd, Serbia}
\email{aleksandra1973@gmail.com}
\author[Karasev]{Roman Karasev}
\thanks{The research by Roman~Karasev leading to these results has received funding from the Federal professorship program grant 1.456.2016/1.4, and from the Russian Foundation for Basic Research grants 18-01-00036 and  19-01-00169.}
\address{Moscow Institute of Physics and Technology, Institutskiy per. 9, Dolgoprudny, Russia 141700\hfill\break
\mbox{\hspace{4mm}}Institute for Information Transmission Problems RAS, Bolshoy Karetny per. 19, Moscow, Russia 127994}
\email{r\_n\_karasev@mail.ru}
\author[Kliem]{Jonathan Kliem}
\address{Inst. Math., FU Berlin, Arnimallee 2, 14195 Berlin, Germany}
\thanks{The research by Jonathan Kliem leading to these results has received funding by the DFG via the Berlin Mathematical School.}
\email{jonathan.kliem@fu-berlin.de}


\begin{abstract}
A union of an arrangement of affine hyperplanes $\HH$ in $\R^d$ is the real algebraic variety associated to the principal ideal generated by the polynomial $p_{\HH}$ given as the product of the degree one polynomials which define the hyperplanes of the arrangement.
A finite Borel measure on  $\R^d$ is bisected by the arrangement of affine hyperplanes $\HH$ if the measure on the ``non-negative side'' of the arrangement $\{x\in\R^d : p_{\HH}(x)\geq 0\}$ is the same as the measure on the ``non-positive'' side of the arrangement $\{x\in\R^d : p_{\HH}(x)\leq  0\}$.

In 2017 Barba, Pilz \& Schnider considered special, as well as modified cases of the following measure partition hypothesis:
For a given collection of $j$ finite Borel measures on $\R^d$ there exists a $k$-element affine hyperplane arrangement that bisects each of the measures into equal halves simultaneously.
They showed that there are simultaneous bisections in the case when $d=k=2$ and $j=4$.
Furthermore, they conjectured that every collection of $j$ measures on $\R^d$ can be simultaneously bisected with a $k$-element affine hyperplane arrangement provided that $d\geq \lceil j/k \rceil$.
The conjecture was confirmed in the case when $d\geq j/k=2^a$ by Hubard and Karasev in 2018.

In this paper we give a different proof of the Hubard and Karasev result using the framework of Blagojevi\'c, Frick, Haase \& Ziegler (2016), based on the equivariant relative obstruction theory of tom Dieck, which was developed for handling the Gr\"unbaum--Hadwiger--Ramos hyperplane measure partition problem.
Furthermore, this approach allowed us to prove even more, that for every collection of $2^a(2h+1)+\ell$ measures on $\R^{2^a+\ell}$, where $1\leq \ell\leq 2^a-1$, there exists a $(2h+1)$-element affine hyperplane arrangement that bisects all of them simultaneously.
Our result was extended to the case of spherical arrangements and reproved by alternative methods in a beautiful way by Crabb in 2020.
\end{abstract}



\date{\today}
\dedicatory{Dedicated to \v{Z}arko Mijajlovi\'c on the occasion of his 70th birthday}

\maketitle

\section{Introduction and statement of main results}
\label{sec : Introduction and statement of main results}

Let $d\geq 1$ be an integer.
An {\em affine hyperplane} in the $d$-dimensional Euclidean space $\R^d$
is determined by a unit vector $u\in S(\R^d)$ in $\R^d$ and a scalar $a\in\R$ as follows:
\[
H_{u,a}:=\{x\in\R^d : \langle x,u\rangle = a\},
\]
where $\langle\cdot,\cdot\rangle$ denotes the standard Euclidean scalar product.
In this description the sets $H_{u,a}$ and $H_{-u,-a}$ coincide.
An {\em oriented affine hyperplane} in $\R^d$ determined by a unit vector $u\in S(\R^d)$ and a scalar $a\in\R$ is the triple $H(u,a):=(H_{u,a},u,a)$.
The set 
of all oriented affine hyperplanes is endowed with a $\Z/2$-action given by the orientation change $H(u,a)\longmapsto H(-u,-a)$.
To each oriented affine hyperplane $H(u,a)$ in $\R^d$ we associate the linear polynomial function $p_{u,a}\colon \R^d\longrightarrow\R$ given by $p_{u,a}(x):= \langle x,u\rangle -a$ for $x\in\R^d$.
In particular, $H_{u,a}=\{x\in\R^d : p_{u,a}(x) = 0\}$.
Furthermore, $p_{u,a}(x)=-p_{-u,-a}(x)$.

\medskip
Let $k\geq 1$ be an integer.
A {\em $k$-element affine hyperplane arrangement} $\HH$ in $\R^d$ is an ordered $k$-tuple of oriented affine hyperplanes in $\R^d$.
To any $k$-element affine hyperplane arrangement
$
\HH=(H(u_1,a_1),\dots,H(u_k,a_k))
$
we associate the polynomial function $p_{\HH}\colon \R^d\longrightarrow\R$ defined by
\[
p_{\HH}(x):= \prod_{i=1}^k p_{u_i,a_i}(x).
\]
The {\em union} of the arrangement $\HH$ in $\R^d$ is the real affine variety
\[
\{x\in\R^d : p_{\HH}(x)=0\}.
\]
A $k$-element affine hyperplane arrangement $\HH=(H(u_1,a_1),\dots,H(u_k,a_k))$ in $\R^d$ is {\em essential} if
\[
H(u_r,a_r)\neq H(u_s,a_s)\qquad\text{and}\qquad H(u_r,a_r)\neq H(-u_s,-a_s),
\]
for all $1\leq r< s \leq k$.
As expected, a $k$-element affine hyperplane arrangement is {\em non-essential} if it is not essential.

\medskip
Let $\mu$ be a {\em nice} measure on $\R^d$, meaning that $\mu$ is a finite Borel measure on $\R^d$ that vanishes on every affine hyperplane in $\R^d$.
A $k$-element arrangement $\HH$  {\em bisects} the family of nice measures  $\mathcal{M}=(\mu_1, \dots, \mu_j)$ if for every $1\leq r\leq j$:
\[
\mu_r \big(\{ x\in\R^d : p_{\HH}(x)\geq 0\}\big)=\mu_r \big(\{ x\in\R^d : p_{\HH}(x)\leq 0\}\big)=\frac{\mu_r(\R^d)}2.
\]

\medskip
In other words, we are looking for an essential affine hyperplane arrangement and a coloring of the connected components of the complement of its union into two colors with the property that no closures of any two components of the same color share a common facet.
This provides a bisection of the space into two parts corresponding to the colors and we ask that this partition bisects every one of the given measures into equal halves.

\medskip
In this paper, motivated by the recent work of Barba, Pilz \& Schnider \cite{PizzaCuttings} we study the set $\Lambda\subseteq\N^3$ of all triples $(d,j,k)$ of positive integers such that for every collection of $j$ nice measures in $\R^d$ there exists a $k$-element affine hyperplane arrangement in $\R^d$ that bisects these measures.
It is not hard to observe that the set  $\Lambda$ has the following property:
\[
(d,j,k)\in\Lambda\quad\Longrightarrow\quad (d',j,k)\in\Lambda \quad\text{for all }d'\geq d.
\]
Furthermore, the ham sandwich theorem is equivalent to the inclusion 
\[
\{(d,j,1) : d\geq j\geq 1 \}\subseteq \Lambda .
\]

\begin{figure}
\centering
\includegraphics[scale=0.8]{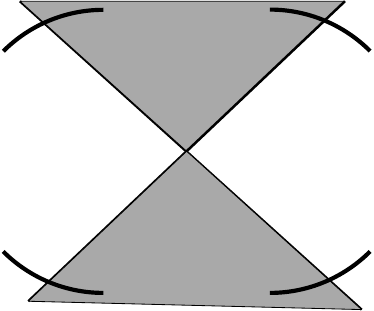}
\caption{\small Illustration of a black and white bisection of four measures on the plane by an essential $2$-element affine hyperplane arrangement.}
\label{fig}
\end{figure}

\medskip
The first description of the set $\Lambda$ follows by considering $j$ pairwise disjoint intervals on a moment curve in $\R^d$ as measures, counting the number of intersection points of a $k$-element affine hyperplane arrangement  with the moment curve (at most $dk$ points) and comparing it with the minimal number of points needed for a bisection of $j$ intervals (at least $j$ points).
Consequently, we get
\[
(d,j,k)\in\Lambda\quad\Longrightarrow\quad dk\geq j.
\]
The idea of considering intervals on a moment curve as measures in the context of the Gr\"unbaum--Hadwiger--Ramos hyperplane measure partition problem originates from the work of Avis \cite{avis1984}, and was further used in this context by Ramos \cite{ramos1996equipartition} and others.
For a detailed review of the Gr\"unbaum--Hadwiger--Ramos hyperplane mass partition problem see for example \cite{BlagojevicFrickHaaseZiegler1} and the references therein.
Thus, it is natural to make the following conjecture, see also \cite[Conj.\,1]{PizzaCuttings}.

\begin{conjecture}
	\label{conjecture}
	Let $d\geq 1$, $j\geq 1$ and $k\geq 1$ be integers.
	If $d\geq \lceil j/k \rceil$, then $(d,j,k)\in\Lambda$.
\end{conjecture}

\medskip
The main result of this paper is derived from the so called ``join configuration space~/ test map scheme'' and an application of two different relative equivariant obstruction theories of Bredon \cite{bredon2006equivariant} and tom Dieck \cite{tom1987transformation}.
The join scheme was introduced for the first time in \cite{blagojevic2011}, while the relative obstruction theory framework for the study of the Gr\"unbaum--Hadwiger--Ramos hyperplane mass partition problem was  developed only in \cite{BlagojevicFrickHaaseZiegler2}.
In particular, in the first part of the theorem we give a different proof of the result by Hubard and Karasev \cite[Thm.\,1]{HubardKarasev}, which in the special case $d=k=2$ and $j=4$ is due to Barba, Pilz \& Schnider \cite[Thm.\,2.2]{PizzaCuttings}.

\begin{theorem}
\label{th : main 2.5}
Let $d\geq 1$, $j\geq 1$ and $k\geq 2$ be integers.
If
\begin{compactenum}[\rm \quad (a)]

\item \label{th : main 2.5 : case 1}
$dk=j$ and $d=2^a$ for some integer $a\geq 0$, or
		
\item \label{th : main 2.5 : case 2}
$(d-\ell)k+\ell=j$, $k$ is odd, $d=2^a+\ell$ for some integers $a\geq 1$, and $1\leq \ell \leq 2^a-1$,
	
\end{compactenum}
then $(d,j,k)\in\Lambda$.
\end{theorem}

\medskip
Thus, Theorem \ref{th : main 2.5}\eqref{th : main 2.5 : case 1} settles Conjecture \ref{conjecture} in the case when
$dk-j=2^a-j=0$, while Theorem \ref{th : main 2.5}\eqref{th : main 2.5 : case 2} gives the positive difference $dk-j=\ell k-\ell$, where $k\geq 3$ is odd, and consequently does not settle the conjecture in any additional case.
The results of Theorem \ref{th : main 2.5} were reproved by Crabb \cite{Crabb2020} in the broader setting of spherical arrangements by intriguing evaluations of pull-backs of twisted Euler classes.

\medskip
In order to illustrate the results of Theorem \ref{th : main 2.5} we fix the parameter $k=3$ and consider the set $\Lambda[k=3]:=\{ (j,d)\in \N^2 : (d,j,3)\in\Lambda\}$.
In Figure~\ref{fig3} we depicted with a black dot for each $j$ the minimal $d$ such that $(j,d) \in \Lambda[k=3]$ as Conjecture \ref{conjecture} claims.
We circled the upper bounds for the dimension $d$ obtained from an application of Theorem \ref{th : main 2.5}\eqref{th : main 2.5 : case 1}.
In grey we circled the improved upper bounds on $d$ derived from Theorem \ref{th : main 2.5}\eqref{th : main 2.5 : case 2}.

\begin{figure}[b]
\centering
\includegraphics[scale=0.61]{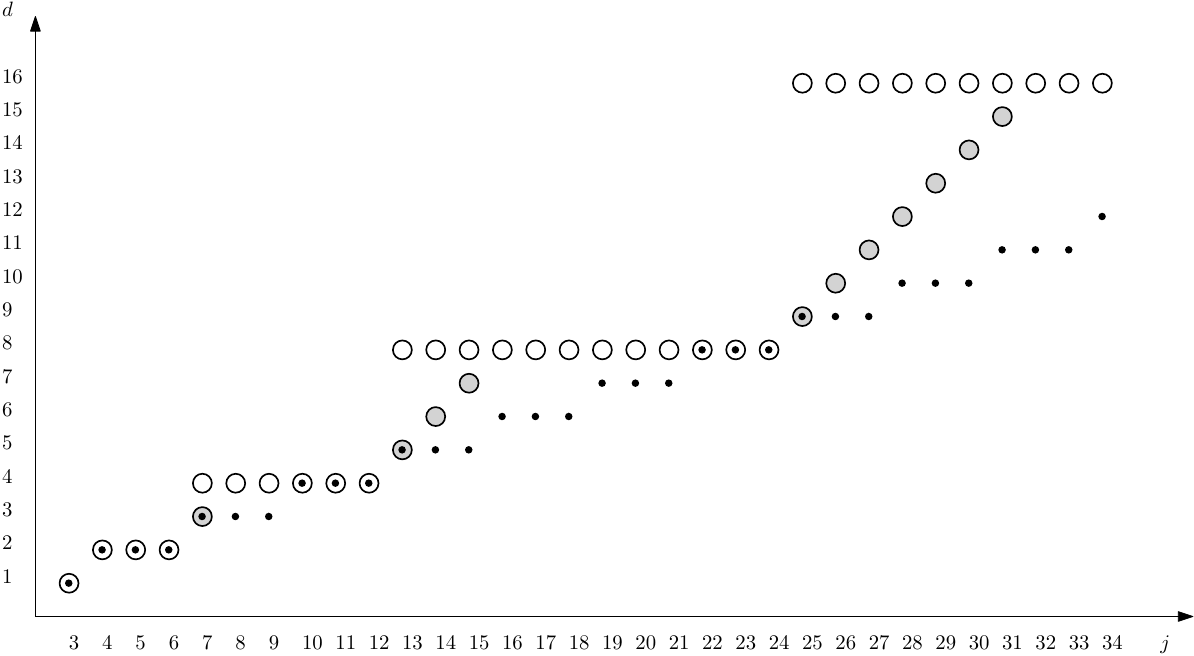}
\caption{\small The shape of the set $\{ (j,d)\in \N^2 : (d,j,3)\in\Lambda\}$ as suggested by Conjecture~\ref{conjecture} (black dot), Theorem~\ref{th : main 2.5}\eqref{th : main 2.5 : case 1} (circled) and Theorem~\ref{th : main 2.5}\eqref{th : main 2.5 : case 2} (circled in grey).}
\label{fig3}
\end{figure}

\medskip
The main result of this paper, stated in Theorem \ref{th : main 2.5}, is proven in the following steps:
\begin{compactitem}[---]
\item The problem regarding the existence of a bisection of a collection of measures in $\R^d$ by a $k$-element affine hyperplane arrangement is connected to the question about the non-existence of specially constructed $\Wk_k$-equivariant maps beween spheres  $S^{(d+1)k-1}\longrightarrow S^{j+k-2}$, see Section \ref{sec : scheme}.
\item The claim of Theorem \ref{th : main 2.5} is obtained as a consequence of the non-existence of $\Wk_k$-equivariant maps $S^{(d+1)k-1}\longrightarrow S^{j+k-2}$ with some specific properties, as explained in Theorem \ref{th : CS/TM}. 
	The non-existence of the relevant map is proved via an application of the equivariant relative obstruction theory of tom Dieck, see Section \ref{sec : proof of the second result}.
\end{compactitem}

\bigskip
\noindent
{\em Acknowledgment.} We would like to thank Alfredo Hubard, Tatiana Levinson and Arkadiy Skopenkov for useful discussions.
The authors thank Matija Blagojevi\'c for his work on the manuscript which resulted in several improvements of the text.
We are also grateful to the referees for valuable suggestions and comments.

\section{From a partition problem to a Borsuk--Ulam type problem}
\label{sec : scheme}
In this section we relate the problem of describing the set $\Lambda\in\N^3$ with a topological problem of the Borsuk--Ulam type.
For that we develop both the product and join configuration space / test map scheme even we apply only the join scheme.
The join scheme can be efficiently used only in combination with the relative equivariant obstruction theory, as demonstrated in \cite{BlagojevicFrickHaaseZiegler2}.

\medskip
The space of all oriented affine hyperplanes in $\R^d$ can be identified with the sphere $S^d=S(\R^{d+1})$ of unit vectors in $\R^{d+1}$ where the north pole $e_{d+1}:=(0,\dots,0,1)$ and the south pole $-e_{d+1}=(0,\dots,0,-1)$ are interpreted as ``extra'' oriented affine hyperplanes at infinity.
To see this place $\R^d$ into $\R^{d+1}$ on ``height one'', that is via the embedding $(x_1,\dots,x_d)\longmapsto (x_1,\dots,x_d,1)$.
Every oriented affine hyperplane $H(u,a)=(H_{u,a},u,a)$ in $\R^d$ spans the uniquely oriented linear hyperplane $H(w(u,a),0)=(H_{w(u,a),0},w(u,a),0)$ in $\R^{d+1}$.
The corresponding unit normal vector $w(u,a)$ determines a point on the sphere $S^d$.
Furthermore, the associated polynomial function $p_{w(u,a),0}\colon \R^{d+1}\longrightarrow\R$, given by $p_{w(u,a),0}(w):=\langle w,w(u,a)\rangle$ for $w\in\R^{d+1}$, restricts on the embedded $\R^d=\R^d\times\{1\}$ to the  polynomial function $p_{u,a}$, that is $p_{u,a}=p_{w(u,a),0}|_{\R^d\times\{1\}}$.
The $\Z/2$ action on the space of all oriented affine hyperplanes given by the change of orientation translates into the antipodal action on the sphere, $w\longmapsto -w$ for $w\in S^{d}$.

\medskip
Following the presentation in \cite[Sec.\,2]{BlagojevicFrickHaaseZiegler2} we consider the following configuration spaces that parameterize all $k$-element affine hyperplane arrangements in $\R^d$:
\begin{compactitem}[\quad ---]
\item the \emph{join configuration space} $X_{d,k}\cong (S^d)^{*k}\cong S(\R^{(d+1)\times k})$ is a sphere of dimension $dk + k-1$, (recall the homeomorphism between unit spheres of Euclidan spaces $S(E_1\oplus E_2)\cong S(E_1) *S(E_2)$), and
\item the \emph{product configuration space} $Y_{d,k}\cong (S^d)^k$.
\end{compactitem}

\medskip
Both configuration spaces are equipped with an action of the group of signed permutations $\Wk_k=(\Z/2)^k \rtimes \Sym_k$.
To define an action on $X_{d,k}$ we recall that its typical element can be presented as formal ordered convex combinations $\lambda_1 w_1 + \dots + \lambda_k w_k$, where $\lambda_i \geq 0,\,\sum_{i=1}^k \lambda_i = 1$ and $w_i \in S^d$.
Now each copy of $\Z/2$ in $(\Z/2)^k\subseteq\Wk_k$ acts antipodally on the appropriate sphere $S^d$, and the symmetric group $\Sym_k\subseteq\Wk_k$ acts by permuting factors in the product.
Explicitly, for $((\beta_1, \dots, \beta_k) \rtimes \tau) \in \Wk_k$ and $\lambda_1 w_1 + \dots + \lambda_k w_k \in X_{d,k}$ we set:
\begin{multline*}
((\beta_1, \dots, \beta_k) \rtimes \tau)\cdot (\lambda_1 w_1 + \dots + \lambda_k w_k) := \\
\lambda_{\tau^{-1}(1)} (-1)^{\beta_1} w_{\tau^{-1}(1)} + \dots +  \lambda_{\tau^{-1}(k)} (-1)^{\beta_k} w_{\tau^{-1}(k)}.
\end{multline*}
Alternatively, we can see the join configuration space $X_{d,k}$ as the unit sphere of the real $\Wk_k$-representation  $\R^{(d+1)\times k}$.
The action of $\Wk_k$ on $\R^{(d+1)\times k}$ we consider is given by:
\[
((\beta_1, \dots, \beta_k) \rtimes \tau)\cdot (u_1, \dots , u_k) :=
\big((-1)^{\beta_1} u_{\tau^{-1}(1)}, \dots ,   (-1)^{\beta_k} u_{\tau^{-1}(k)}\big),
\]
for $((\beta_1, \dots, \beta_k) \rtimes \tau) \in \Wk_k$ and $(u_1, \dots , u_k)\in \R^{(d+1)\times k}$.

\medskip
The subspace
\[
\big\{ \tfrac1k w_1+\dots+\tfrac1k w_k \in X_{d,k} : (w_1,\dots, w_k)\in  Y_{d,k}\big\}
\]
of the join $X_{d,k}$ is homeomorphic to  $Y_{d,k}$ and moreover $\Wk_k$-invariant.
Thus we identify it with $Y_{d,k}$, and the restriction action from $X_{d,k}$ induces an $\Wk_k$-action on $Y_{d,k}$.
For $k\geq 2$ action of $\Wk_k$ on both $X_{d,k}$ and $Y_{d,k}$ is not free.
The subspaces of points of $X_{d,k}$ and $Y_{d,k}$ with non-trivial stabilizers with respect to the $\Wk_k$-action are
\begin{multline*}
X_{d,k}^{>1}:  = \{\lambda_1 w_1 + \dots + \lambda_k w_k \in X_{d,k} :
\lambda_1\cdots\lambda_k=0,\text{  or  }\lambda_s=\lambda_r \\
\text{ with }w_s=\pm w_r\text{ for some }1\leq s<r\leq k\},
\end{multline*}
and
\[
Y_{d,k}^{>1} : = \{(w_1,\ldots, w_k)\in Y_{d,k}: w_s=\pm w_r\text{ for some }1\leq s<r\leq k\}.
\]
For future use we point out the  subspace of $X_{d,k}^{>1}$ given by:
\[
(X_{d,k}^{>1})':=  \{\lambda_1 w_1 + \dots + \lambda_k w_k\in X_{d,k}^{>1} : \lambda_1\cdots\lambda_k=0 \}.
\]

\medskip
Let $V\cong\R$ be the real $1$-dimensional $\Wk_k$-representation with action defined to be antipodal  for every copy of $\Z/2$ in $(\Z/2)^k\subseteq\Wk_k$, and trivial for every element of the symmetric group $\Sym_k\subseteq\Wk_k$.
More precisely, when $((\beta_1, \dots, \beta_k) \rtimes \tau) \in \Wk_k$ and $v\in V$ we have
\[
((\beta_1, \dots, \beta_k) \rtimes \tau)\cdot v :=
(-1)^{\beta_1} \cdots (-1)^{\beta_k}\, v.
\]

\medskip
Next consider the vector space $\R^k$ and its vector subspace
\[
W_k=\big\{(y_1,\ldots,y_k)\in\R^k : \sum_{i=1}^k y_i=0\big\}.
\]
The group $\Wk_k$ acts on $\R^k$ by permuting coordinates, that is,  for $((\beta_1, \dots, \beta_k) \rtimes \tau) \in \Wk_k$ and $(y_1,\ldots,y_k)\in\R^k$ we get
\begin{equation}
\label{eq:action_W_k}
((\beta_1, \dots, \beta_k) \rtimes \tau) \cdot (y_1,\ldots,y_k) : = (y_{\tau^{-1}(1)},\ldots,y_{\tau^{-1}(k)}).
\end{equation}
The subspace $W_k\subseteq\R^k$ is $\Wk_k$-invariant, and therefore $W_k$ is an $\Wk_k$-subrepresentation of $\R^k$.

\medskip
Now, to an ordered collection $\mathcal{M}=(\mu_1, \dots, \mu_j)$ of nice measures on $\R^d$ we will associate two continuous $\Wk_k$-equivariant maps $\Phi_{\mathcal{M}}$ and $\Psi_{\mathcal{M}}$.

\medskip
First, we define the continuous map
\[
\Phi_{\mathcal{M}} \colon Y_{d,k} \longrightarrow V^{\oplus j}
\]
to be the unique continuous extension of the map $(S^d{\setminus}\{e_{d+1},-e_{d+1}\} )^k \longrightarrow V^{\oplus j}$
given by
\begin{multline}\label{map:to extend}
\HH=(H(u_1,a_1),\dots,H(u_k,a_k))=(w(u_1,a_1),\dots,w(u_k,a_k))\longmapsto \\
\Big( \mu_i (\{ x\in\R^d : p_{\HH}(x)\geq 0\})-\mu_i (\{ x\in\R^d : p_{\HH}(x)\leq 0\}) \Big)_{i \in \{1,\dots, j\}}.
\end{multline}
Indeed, the map \eqref{map:to extend} is the restriction of the continuous function $(S^d)^k \longrightarrow V^{\oplus j}$ defined by
\begin{multline*}
(w_1,\dots,w_k)\longmapsto
\Big( \mu_i (\{ (x,1)\in\R^{d+1} : p_{w_1,\dots,w_k}(x,1)\geq 0\})- \\  \mu_i (\{ (x,1)\in\R^{d+1} : p_{w_1,\dots,w_k}(x,1)\leq 0\}) \Big)_{i \in \{1,\dots, j\}}.
\end{multline*}
Here $p_{w_1,\dots,w_k}\colon\R^{d+1}\longrightarrow\R$ is the continuous function $p_{w_1,\dots,w_k}(w):= \prod_{i=1}^k \langle w,w_i\rangle$.

\medskip
The map $\Phi_{\mathcal{M}}$ is $\Wk_k$-equivariant with respect to the already introduced actions on $Y_{d,k}$ and $V$, assuming the diagonal action on $V^{\oplus j}$.

\medskip
The key property of the map $\Phi_{\mathcal{M}}$ is that \emph{the $k$-element affine hyperplane arrangement $\HH$ in $\R^d$ bisects all the measures from the collection $\mathcal{M}$ if and only if $\Phi_{\mathcal{M}} (\HH) = 0\in V^{\oplus j}$}.

\medskip
The second continuous map we consider
\[
\Psi_{\mathcal{M}} \colon X_{d,k} \longrightarrow W_k\oplus V^{\oplus j}
\]
is defined as follows:
\begin{equation}\label{eq : test map def}
\lambda_1 w_1 +\cdots + \lambda_k w_k \longmapsto \big(\lambda_1-\tfrac{1}{k},\ldots,\lambda_k-\tfrac{1}{k}\big)\oplus \Big((\lambda_1\cdots\lambda_k) \cdot \Phi_{\mathcal{M}} (w_1, \dots, w_k)\Big).
\end{equation}
It is important to notice that the map we have just defined $\Psi_{\mathcal{M}}$ does not depend on the collection $\mathcal{M}$ when considered on the subspace $(X_{d,k}^{>1})'$.
Indeed, if
\[
\lambda_1 w_1 + \dots + \lambda_k w_k\in (X_{d,k}^{>1})',
\]
then
\[
\Psi_{\mathcal{M}}(\lambda_1 w_1 + \dots + \lambda_k w_k)
=
\big(\lambda_1-\tfrac{1}{k},\ldots,\lambda_k-\tfrac{1}{k}\big)\oplus 0\in W_k\oplus V^{\oplus j}.
\]
The map $\Psi_{\mathcal{M}}$ is also $\Wk_k$-equivariant.

\medskip
Similarly, {\em the $k$-element affine hyperplane arrangement
\[
\HH=(H(u_1,a_1),\dots,H(u_k,a_k))=(w(u_1,a_1),\dots,w(u_k,a_k))=(w_{1},\dots,w_{k})
\]
in $\R^d$ bisects all the measures from the collection $\mathcal{M}$ if and only if
\[
\Psi_{\mathcal{M}} \big(\tfrac1k w_{1}+\dots+\tfrac1k  w_{k}\big)= 0\oplus 0\in W_k\oplus V^{\oplus j}.
\]
}

\medskip
From the construction of the $\Wk_k$-equivariant maps $\Phi_{\mathcal{M}}$ and $\Psi_{\mathcal{M}}$ we have deduced the following facts; for a similar construction consult \cite[Prop.\,2.1]{BlagojevicFrickHaaseZiegler2}.

\begin{proposition}
\label{prop:CS/TM}
Let $d\geq 1$, $j\geq 1$ and $k\geq 1$ be integers.
\begin{compactenum}[\rm \quad (a)]
\item \label{item:test-map-zeros}
Let $\mathcal{M}$ be a collection of $j$ nice measures on $\R^d$, and let
\[
\Phi_{\mathcal{M}} \colon Y_{d,k} \longrightarrow V^{\oplus j}
\quad\text{and}\quad
\Psi_{\mathcal{M}} \colon X_{d,k} \longrightarrow W_k\oplus V^{\oplus j}
\]
be the $\Wk_k$-equivariant maps defined above.
If
\[
0\in \im\Phi_{\mathcal{M}}
\qquad\text{or}\qquad
0\in \im\Psi_{\mathcal{M}},
\]
then there exists a $k$-element affine hyperplane arrangement bisecting all the measures in $\mathcal{M}$.
\item \label{item:existence-of-maps}
If there is no $\Wk_k$-equivariant map of either type
\[
Y_{d,k} \longrightarrow S(V^{\oplus j})
\quad\text{or}\quad
X_{d,k} \longrightarrow S(W_k\oplus V^{\oplus j}),
\]
then $(d,j,k)\in\Lambda$.
\end{compactenum}
\end{proposition}

\medskip
The following essential property of the constructed $\Wk_k$-equivariant map  $\Psi_{\mathcal{M}}$ needs a modified approach unlike the one used in \cite[Prop.\,2.2]{BlagojevicFrickHaaseZiegler2}.

\begin{proposition}
\label{prop:homotopy}
Let $d\geq 2$, $j\geq 1$ and $k\geq 2$ be integers with $j\geq d(k-1)+2$.
Let $\mathcal{M}=(\mu_1,\dots,\mu_j)$ and $\mathcal{M'}=(\mu_1',\dots,\mu_j')$ be collections of nice measures on $\R^d$ such that no non-essential $k$-element affine hyperplane arrangement bisects all of them.
Then
\begin{compactenum}[\rm \quad (a)]
\item  $0\notin \im \Psi_{\mathcal{M}}|_{X_{d,k}^{>1}}$, \label{item:avoiding-zero}
\item $\Psi_{\mathcal{M}}|_{(X_{d,k}^{>1})'}=\Psi_{\mathcal{M'}}|_{(X_{d,k}^{>1})'}$, \label{item:eqaul on a subspace}  and
\item $\Psi_{\mathcal{M}}|_{X_{d,k}^{>1}}$ and $\Psi_{\mathcal{M'}}|_{X_{d,k}^{>1}}$ are $\Wk_k$-homotopic as maps
\[
X_{d,k}^{>1}\longrightarrow (W_k\oplus V^{\oplus j}){\setminus}\{0\}
\]
which restrict on the subspace $(X_{d,k}^{>1})'$ to the map given by
\[
\lambda_1w_1+\dots+\lambda_kw_k \longmapsto \big(\lambda_1-\tfrac{1}{k},\ldots,\lambda_k-\tfrac{1}{k}\big)\oplus 0,
\]
where $\lambda_1w_1+\dots+\lambda_kw_k\in (X_{d,k}^{>1})'$ and $ \big(\lambda_1-\tfrac{1}{k},\ldots,\lambda_k-\tfrac{1}{k}\big)\oplus 0\in (W_k\oplus V^{\oplus j}){\setminus}\{0\}$.
\label{item:psi-homotopy}
\end{compactenum}
\end{proposition}

\medskip
The previous proposition is the special, easy to state, $\ell=0$ case of a stronger statement which works on invariant subcomplexes of $X_{d,k}$, and therefore on $X_{d,k}$ itself; see Proposition~\ref{prop:specialcase}.
Hence, we only prove the more general result.

\medskip
In the upcoming proposition we use the $\Wk_k$-CW structure on the join configuration space $X_{d,k}$ developed in \cite[Sec.\,3]{BlagojevicFrickHaaseZiegler2}, and reviewed in Section \ref{sec : appendix} of this paper.

\begin{proposition}\label{prop:specialcase}
Let $d\geq 2$, $j\geq 1$ and $k\geq 2$ be integers, and let $1 \leq \ell \leq d-1$ be an integer with $(d - \ell )(k-1) + 2 + \ell \leq j$.
Let $Z := \Wk_k\cdot \overline{\theta}$ denote the full $\Wk_k$-orbit of the closure of the cell $\theta := D_{1+\ell,\ldots,1+\ell,1}^{+,\ldots,+,+}(1,2,\ldots,k)$, and let $Z^{>1} := Z \cap X_{d,k}^{>1}$ and  $(Z^{>1})':=Z\cap (X_{d,k}^{>1})'$.
Furthermore, let $\mathcal{M}=(\mu_1,\dots,\mu_j)$ and $\mathcal{M'}=(\mu_1',\dots,\mu_j')$ be collections of  nice measures on $\R^d$ such that no $k$-element affine hyperplane arrangement parameterized by $Z^{>1}$ bisects them.
Then
\begin{compactenum}[\rm \quad (a)]
\item $0\notin \im \Psi_{\mathcal{M}}|_{Z^{>1}}$,
and
\item $\Psi_{\mathcal{M}}|_{Z^{>1}}$ and $\Psi_{\mathcal{M'}}|_{Z^{>1}}$ are $\Wk_k$-homotopic as maps
\[
Z^{>1}\longrightarrow (W_k\oplus V^{\oplus j}){\setminus}\{0\}
\]
which restrict on the subspace $(Z^{>1})'$ to the map given by 
\begin{equation}\label{eq : restricted map}
	\lambda_1w_1+\dots+\lambda_kw_k \longmapsto \big(\lambda_1-\tfrac{1}{k},\ldots,\lambda_k-\tfrac{1}{k}\big)\oplus 0,
\end{equation}
where $\lambda_1w_1+\dots+\lambda_kw_k\in (Z^{>1})'$ and $ \big(\lambda_1-\tfrac{1}{k},\ldots,\lambda_k-\tfrac{1}{k}\big)\oplus 0\in (W_k\oplus V^{\oplus j}){\setminus}\{0\}$.

\end{compactenum}
\end{proposition}
\begin{proof}
The first statement  follows directly from the assumption that no $k$-element affine hyperplane arrangement parameterized by $Z^{>1}$ bisects $\mathcal{M}$.

\medskip
From the assumption on the collections of measures $\mathcal{M}$ and $\mathcal{M'}$ and the first part of the proposition we have that $0\notin \im \Psi_{\mathcal{M}}|_{Z^{>1}}$ and $0\notin \im \Psi_{\mathcal{M'}}|_{Z^{>1}}$.
Consequently the maps $\Psi_{\mathcal{M}}|_{Z^{>1}}$ and $\Psi_{\mathcal{M'}}|_{Z^{>1}}$ can be considered as $\Wk_k$-equivariant maps $Z^{>1}\longrightarrow (W_k\oplus V^{\oplus j}){\setminus}\{0\}$.
Furthermore, from the definition of the test map \eqref{eq : test map def} follows that the maps $\Psi_{\mathcal{M}}|_{(Z^{>1})'}=\Psi_{\mathcal{M'}}|_{(Z^{>1})'}$ coincide with the map \eqref{eq : restricted map}.

\medskip
In order to prove the second statement we need to construct an $\Wk_k$-equivariant homotopy
\[
F\colon Z^{>1}\times I \longrightarrow (W_k\oplus V^{\oplus j}){\setminus}\{0\}
\]
between the maps $\Psi_{\mathcal{M}}|_{Z^{>1}}$ and $\Psi_{\mathcal{M'}}|_{Z^{>1}}$.
Here $I$ denotes the unit interval $[0,1]$.
This will be done using a slight extension of the equivariant obstruction theory of Bredon \cite[Ch.\,II]{bredon2006equivariant} as presented in \cite[Ch.\,I.5]{MayEtAl1996} because the obstruction theory of tom Dieck \cite[Sec.\,II.3]{tom1987transformation} cannot be used in this situation.
Indeed, notice that no point in $Z^{>1}$ has a trivial stabilizer.

\medskip
For simplicity, we denote by
\[
K:=Z^{>1}\times I
\qquad\text{and}\qquad
L:=Z^{>1} \times \{0\}\,\cup\, Z^{>1}\times \{1\}\, \cup \, (Z^{>1})' \times I.
\]
Define $F_{-1}\colon L\longrightarrow (W_k\oplus V^{\oplus j}){\setminus}\{0\}$ by
\begin{eqnarray*}
	F_{-1}|_{Z^{>1}\times \{0\}}&:=&\Psi_{\mathcal{M}}|_{Z^{>1}},\\
	F_{-1}|_{Z^{>1}\times \{1\}}&:=&\Psi_{\mathcal{M'}}|_{Z^{>1}},\\
	F|_{(Z^{>1})'\times\{t\}}&=&\Psi_{\mathcal{M}}|_{(Z^{>1})'}=\Psi_{\mathcal{M'}}|_{(Z^{>1})'}, \quad\text{for all }t\in I.
\end{eqnarray*}
Our aim is to extend the $\Wk_k$-equivariant map $F_{-1}$ to an  $\Wk_k$-equivariant map $F\colon K \longrightarrow (W_k\oplus V^{\oplus j}){\setminus}\{0\}$ extending it one skeleton at a time.

\medskip
Since $(W_k\oplus V^{\oplus j}){\setminus}\{0\}$ is non-empty, and in addition, for every subgroup $G$ of $\Wk_k$, the following implication holds:
\[
 (K{\setminus}L)^G\neq\emptyset \Longrightarrow ((W_k\oplus V^{\oplus j}){\setminus}\{0\})^G\neq\emptyset
\]
we can extend $F_{-1}$ to the $0$-skeleton of $K$ obtaining an $\Wk_k$-equivariant map $F_0$.
Assume that we have defined a $\Wk_k$-equivariant map $F_r$ on the $r$-th skeleton of $K$.

\medskip
The obstructions for the extension of the map $F_r$ to the next skeleton live in the Bredon cohomology group
\[
	H_{\Wk_k}^{r+1}\big(K,L;\widetilde{\omega}_{r}((W_k\oplus V^{\oplus j}){\setminus}\{0\})\big),
\]
where $0\leq r\leq \dim K-1 = \dim Z^{>1}$.
Here  $\widetilde{\omega}_r((W_k\oplus V^{\oplus j}){\setminus}\{0\})\colon \mathcal{O}_{\Wk_k}\longrightarrow\mathrm{Ab}$ denotes the generic coefficient system.
That is a contravariant functor from the category of canonical objects $\mathcal{O}_{\Wk_k}$ of the group $\Wk_k$ associated to the pair $(K,L)$ into the category of Abelian groups given on objects by
\[
\widetilde{\omega}_r((W_k\oplus V^{\oplus j}){\setminus}\{0\})(\Wk_k/G)=\pi_r(((W_k\oplus V^{\oplus j}){\setminus}\{0\})^G,y_0^G),
\]
where $(K{\setminus}L)^G\neq\emptyset$.
Here, for every subgroup $G$ of $\Wk_k$ with the property that $((W_k\oplus V^{\oplus j}){\setminus}\{0\})^G \neq 0$, we chose a $G$-fixed base point $y_0^G$ in  $((W_k\oplus V^{\oplus j}){\setminus}\{0\})^G$ such that for every subgroup $H_1$ contained in the conjugacy class $gH_2g^{-1}$ of another subgroup $H_2$ holds $gy_0^{H_2} = y_0^{H_1}$.
Such a choice can be made beacuse $\Wk_k$ is finite.
For a detailed account of all relevant  notions see \cite[Ch.\,I.5]{MayEtAl1996} \cite[Ch.\,I.4]{bredon2006equivariant}.

\medskip
Let $D_{i_1,\dots,i_k}^{s_1,\dots,s_k}(\sigma)\times I$ be an arbitrary $r+1$ cell of $K\setminus L$. The cocycle corresponding to $D_{i_1,\dots,i_k}^{s_1,\dots,s_k}(\sigma)\times I$ will have coefficients in $\pi_r(((W_k\oplus V^{\oplus j}){\setminus}\{0\})^G,y_0^G)$, where $G \subseteq \Wk_k$ is the stabilizer group of the cell $D_{i_1,\dots,i_k}^{s_1,\dots,s_k}(\sigma)\times I$.

\medskip
As $D_{i_1,\dots,i_k}^{s_1,\dots,s_k}(\sigma)\times I$ lies in $K \setminus L$, the first  $1\leq i_1 \leq d+1$ and second $k-1$ of the positive integers $i_1,\dots,i_k$ are less than or equal to $1+\ell$.
Furthermore, as $G$ is the stabilizing group of the cell $D_{i_1,\dots,i_k}^{s_1,\dots,s_k}(\sigma)$ and $i_1 \leq d+1$, $G$ is not a subgroup of the defining subgroup $(\Z/2)^k$ of the group $\Wk_k=(\Z/2)^k\rtimes\Sym_k$.
Let $(\beta_1, \dots, \beta_k) \rtimes \tau \in \Wk_k$ be an element that fixes the cell $D_{i_1,\dots,i_k}^{s_1,\dots,s_k}(\sigma)$, that is
\[
(\beta_1, \dots, \beta_k) \rtimes \tau \cdot D_{i_1,\dots,i_k}^{s_1,\dots,s_k}(\sigma) = D_{i_1,\dots,i_k}^{(-1)^{\beta_1}s_1,\dots,(-1)^{\beta_k}s_k}(\tau\sigma)=D_{i_1,\dots,i_k}^{s_1,\dots,s_k}(\sigma).
\]
Consequently, we have that
\begin{compactitem}[\quad ---]
\item $(-1)^{\beta_q}s_q = s_{\tau^{-1}(q)}$ for all $1 \leq q \leq k$, and
\item $i_r = d+2$ for each $\tau(q) < r \leq q$ resp.\ $q < r \leq \tau(q)$ for all $1 \leq q \leq k$ with $\tau(q)\neq q$.
\end{compactitem}
In particular,  $(-1)^{\beta_1} \cdots (-1)^{\beta_k} = 1$ and so $(V^{\oplus j})^G = V^{\oplus j}$.

\medskip
Next, the dimension of the cell $D_{i_1,\dots,i_k}^{s_1,\dots,s_k}(\sigma)$ can be estimated as follows.
Let us first introduce  $z=z_{(i_1,\dots,i_k)}:=\#\{r : 1 \leq r \leq k \text{ and } i_r = d+2\}$.
Notice that $1\leq z\leq k-1$.
Then
\begin{multline*}
	r=\dim D_{i_1,\dots,i_k}^{s_1,\dots,s_k}(\sigma) = (d+1)k - 1 - \sum_{q = 1}^k (i_q - 1)\\ \leq (d+1)k - 1 - \ell(k-1) - z(d+1 - \ell).
\end{multline*}
On the other hand,
\[
\dim (W_k\oplus V^{\oplus j})^G  = k-1-z+j.
\]
From the assumptions that $1 \leq \ell \leq d-1$ and $(d - \ell )(k-1) + 2 + \ell \leq j$ we get that
\begin{equation}
	\label{eq : inequality - 01}
	r+1=\dim (D_{i_1,\dots,i_k}^{s_1,\dots,s_k}(\sigma)\times I ) \leq \dim S((W_k\oplus V^{\oplus j})^G).
\end{equation}
This conclusion follows from a direct verification of the inequality
\[
(d+1)k - \ell(k-1) - z(d+1 - \ell) \leq k-2-z+j,
\]
or more precisely the inequality
\[
(d - \ell )(k-1) + 2 + \ell \geq (d+1)k - \ell(k-1) - z(d+1 - \ell)-k+2+z.
\]

\medskip
Now, the relevant generic coefficient system $\widetilde{\omega}_r((W_k\oplus V^{\oplus j}){\setminus}\{0\})(\Wk_k/G)$ vanishes.
Indeed, from inequality \eqref{eq : inequality - 01} it follows that 
\begin{multline*}
\widetilde{\omega}_r((W_k\oplus V^{\oplus j}){\setminus}\{0\})(\Wk_k/G)=
\pi_r(((W_k\oplus V^{\oplus j}){\setminus}\{0\})^G,y_0^G)\\
\cong \pi_r(S((W_k\oplus V^{\oplus j}){\setminus}\{0\})^G,\tfrac{1}{\|y_0^G\|}y_0^G)\overset{\eqref{eq : inequality - 01}}{=}0.
\end{multline*}
Thus, the Bredon cohomology group $H_{\Wk_k}^{r+1}(K,L;\widetilde{\omega}_{r}((W_k\oplus V^{\oplus j}){\setminus}\{0\}))$ also vanishes.
Consequently all obstructions, in all dimensions $r+1 \leq  \dim Z^{>1}$, vanish.
Thus,  the $\Wk_k$-equivariant map $F_{-1}\colon L\longrightarrow (W_k\oplus V^{\oplus j}){\setminus}\{0\}$ extends to an $\Wk_k$-equivariant map $F\colon Z^{>1}\times I \longrightarrow (W_k\oplus V^{\oplus j}){\setminus}\{0\}$, that is to an $\Wk_k$-homotopy between the maps $\Psi_{\mathcal{M}}|_{Z^{>1}}$ and $\Psi_{\mathcal{M'}}|_{Z^{>1}}$.
\end{proof}
\begin{remark}
In general, for a finite group $G$ the category of canonical object  $\mathcal{O}_G$ of $G$ consists of all sets of left cosets $G/H$ as objects, where $H$ is a subgroup of $G$, and all $G$-equivariant maps $G/H_1\longrightarrow G/H_2$ between them as morphisms.
Here the action of $G$ on the objects is assumed to be induced by the left translations; see \cite[Ch.\,I.3]{bredon2006equivariant}.  
A generic coefficient system of group $G$ is any contravariant functor $\omega\colon \mathcal{O}_G\longrightarrow\mathrm{Ab}$ from the category of canonical object into the category of abelian groups; for more details consult \cite[Ch.\,I.4]{bredon2006equivariant}.
\end{remark}

\medskip
Now we combine the criterion stated in Proposition~\ref{prop:CS/TM}\,(\ref{item:existence-of-maps}) and the observations from Proposition~\ref{prop:homotopy} and Proposition~\ref{prop:specialcase} into a theorem.
In the following, $\nu$ denotes the radial $\Wk_k$-equivariant deformation retraction
\[
\nu\colon (W_k\oplus V^{\oplus j}){\setminus}\{0\}\longrightarrow S(W_k\oplus V^{\oplus j}).
\]

\begin{theorem}~
\label{th : CS/TM}
\begin{compactenum}[\rm \quad (a)]
\item \label {th : CS/TM - main case}
Let $d\geq 1$, $j\geq 1$ and $k\geq 2$ be integers with $d(k-1) + 2 \leq j$, and let $\mathcal{M}$ be any collection of $j$ nice measures on~$\R^d$ such that no non-essential $k$-element affine hyperplane arrangement bisects them.
If there is no $\Wk_k$-equivariant map
	\[
		X_{d,k} \longrightarrow S(W_k\oplus V^{\oplus j})
	\]
	whose restriction on $X_{d,k}^{>1}$ is $\Wk_k$-homotopic to
	$\nu\circ\Psi_{\mathcal{M}}|_{X_{d,k}^{>1}}$, then $(d,j,k)\in\Lambda$.
\item \label{th : CS/TM - ell > 0}
Let $d\geq 1$, $j\geq 1$ and $k\geq 2$ be integers, and let $0 \leq \ell \leq d-1$ be an integer such that  $(d - \ell )(k-1) + 2 + \ell \leq j$.
Set $Z := \Wk_k\cdot \overline{\theta}$ to be the $\Wk_k$-orbit of the closure of the cell $\theta := D_{1+\ell,\ldots,1+\ell,1}^{+,\ldots,+}(1,2,\ldots,k)$, and $Z^{>1} := Z \cap X_{d,k}^{>1}$.
If, for a collection $\mathcal{M}$ of $j$ nice measures on~$\R^d$ such that $0\notin \im \Psi_{\mathcal{M}}|_{Z^{>1}}$,
there is no $\Wk_k$-equivariant map
	\[
		Z \longrightarrow S(W_k\oplus V^{\oplus j})
	\]
	whose restriction on $Z^{>1}$ is $\Wk_k$-homotopic to
	$\nu\circ\Psi_{\mathcal{M}}|_{Z^{>1}}$, then $(d,j,k)\in\Lambda$.
\end{compactenum}
\end{theorem}

\section{Proof of Theorem \ref{th : main 2.5}}
\label{sec : proof of the second result}
From this point on we fix an $\Wk_k$-CW structure on the sphere $X_{d,k}$ to be the one introduced and described in \cite[Sec.\,3]{BlagojevicFrickHaaseZiegler2} and reviewed in the appendix of this paper, Section \ref{sec : appendix}.

\medskip
Let $d\geq 1$, $j\geq 1$ and $k\geq 2$ be integers.
We want to prove that if one of the conditions (1)-(2) of Theorem \ref{th : main 2.5} is satisfied, then for every collection $\mathcal{M}$ of $j$ nice measures in $\R^d$ there exists a $k$-element affine hyperplane arrangement in $\R^d$ that bisects each of the measures.
For this, according to Theorem~\ref{th : CS/TM}\eqref{th : CS/TM - main case}, in case $j = dk$ it suffices to prove that there is no $\Wk_k$-equivariant map
\[
		X_{d,k} \longrightarrow S(W_k\oplus V^{\oplus j}),
\]
whose restriction on $X_{d,k}^{>1}$ is $\Wk_k$-homotopic to $\nu\circ\Psi_{\mathcal{M}_0}|_{X_{d,k}^{>1}}$, where $\mathcal{M}_0$ is a fixed collection of $j$ nice measures on $\R^d$ such that no non-essential $k$-element affine hyperplane arrangement bisects them.

\medskip
Alternatively, according to Theorem~\ref{th : CS/TM}\eqref{th : CS/TM - ell > 0}, we may consider $Z := \Wk_k\cdot\overline{\theta}$ to be the full $\Wk_k$-orbit of the closure of the cell $\theta := D_{1+\ell,\ldots,1+\ell,1}^{+,\ldots,+,+}(1,2,\ldots,k)$ for some $0 \leq \ell \leq d$ such that $(d - \ell)(k-1) + 2 + \ell \leq j$.
If $j = k(d - \ell) + \ell$ this is indeed satisfied.
As before, set  $Z^{>1} := Z \cap X_{d,k}^{>1}$.
Then it suffices to prove that there is no $\Wk_k$-equivariant map
\[
		Z \longrightarrow S(W_k\oplus V^{\oplus j}),
\]
whose restriction on $Z^{>1}$ is $\Wk_k$-homotopic to $\nu\circ\Psi_{\mathcal{M}_0}|_{Z^{>1}}$, where $\mathcal{M}_0$ is a fixed collection of $j$ nice measures on $\R^d$ such that no $k$-element affine hyperplane arrangement parameterized by $Z^{>1}$ bisects them.

\medskip
Therefore, our proof of Theorem \ref{th : main 2.5} follows directly from the following two propositions.
The first proposition gives divisibility criterions for the nonexistence of an $\Wk_k$-equivariant map $Z \longrightarrow S(W_k\oplus V^{\oplus j})$ with required properties.

\begin{proposition}
\label{th : main 3}
Let $d\geq 1$, $j\geq 1$ and $k\geq 2$ be integers.
\begin{compactenum}[\rm\quad (a)]
\item \label{th : main 3 : case 1}
If $dk=j$, and $\frac{1}{k!}{ dk \choose {d,\dots,d}}$ is odd, and $Z := X_{d,k}$, $Z^{>1}: = X_{d,k}^{>1}$, or
\item \label{th : main 3 : case 2}
if there exists an integer $\ell$ such that  $1 \leq \ell \leq  d-1$,  $(d-\ell)k+\ell=j$ and ${(d-\ell)k+\ell \choose d}\tfrac{1}{(k-1)!}{(d-\ell)(k-1) \choose {d-\ell,\dots,d-\ell}}$ is odd, and $Z := \Wk_k\cdot\theta$, $Z^{>1} := Z \cap X_{d,k}^{>1}$ where $\theta:=D_{1+\ell,\ldots,1+\ell,1}^{+,\ldots,+,+}(1,2,\ldots,k)$,
\end{compactenum}
then there is no $\Wk_k$-equivariant map
\begin{equation}
\label{eq : eq-map-2}
Z \longrightarrow  S(W_k\oplus V^{\oplus j}),
\end{equation}
whose restriction to $Z^{>1}$ is $\Wk_k$-homotopic to
	$\nu\circ\Psi_{\mathcal{M}_0}|_{Z^{>1}}$, where $\mathcal{M}_0$ is some fixed collection of $j$ nice measures on $\R^d$ such that that no $k$-element affine hyperplane arrangement parameterized by $Z^{>1}$ bisects them.

\end{proposition}

\medskip
The proof of Proposition \ref{th : main 3}\eqref{th : main 3 : case 1} will actually give us more, since by construction of $Z \subseteq X_{d,k}$  the existence of the $\Wk_k$-equivariant map \eqref{eq : eq-map-2} depends only on the primary obstruction.
In the case when $Z=X_{d,k}$ we have the following equivalence.

\begin{corollary}
	\label{cor : existence if and only if}
	Let $d\geq 1$, $j\geq 1$ and $k\geq 2$ be integers, and let $dk=j$.
	Then $\frac{1}{k!}{ dk \choose {d,\dots,d}}$ is even if and only if there exists an $\Wk_k$-equivariant map
	\[X_{d,k} \longrightarrow  S(W_k\oplus V^{\oplus j})\]
	whose restriction on $X_{d,k}^{>1}$ is $\Wk_k$-homotopic to
	$\nu\circ\Psi_{\mathcal{M}_0}|_{X_{d,k}^{>1}}$, where $\mathcal{M}_0$ is a certain fixed collection of $j$ nice measures on $\R^d$ such that no non-essential $k$-element affine hyperplane arrangement bisects them.
    \end{corollary}
\medskip
The second proposition shows when the divisibility criterions of Proposition \ref{th : main 3} are satisfied.
More precisely, it shows that the assumptions of Theorem \ref{th : main 2.5} are equivalent to the divisibility criterions in  Proposition \ref{th : main 3}.
The first case of the proposition is the content of \cite[Lem.\,5]{HubardKarasev}.
In the second case we restrict $\ell$ to $2\leq 2 \ell \leq d-1$ as the case $2\ell \geq d$ does not yield any new bound.

\begin{proposition}\label{lem : number theory}
	Let $d\geq 1$, $j\geq 1$ and $k\geq 2$ be integers.
	\begin{compactenum}[\rm\quad (a)]
	\item Let $dk=j$.
	Then $\frac{1}{k!}{ dk \choose {d,\dots,d}}$ is odd if and only if $d=2^a$ for some integer $a\geq 0$.
	\item Let $(d-\ell)k+\ell=j$ and $2\leq 2 \ell \leq d-1$.
	Then ${(d-\ell)k+\ell \choose d}\cdot\frac{1}{(k-1)!}{(d-\ell)(k-1) \choose {d-\ell,\dots,d-\ell}}$ is odd if and only if $k$ is odd and
	 $d=2^a+\ell$ for some integer $a\geq 1$.
	\end{compactenum}
\end{proposition}

\medskip
The first part of the proposition is the content of \cite[Lem.\,5]{HubardKarasev}.
In the second part of the proposition we restrict to such $\ell$ where $2\leq 2 \ell \leq d-1$ as the case $2\ell \geq d$ does not yield any new bounds.

\medskip
Now, using Theorem~\ref{th : CS/TM}, Propositions \ref{th : main 3} and Proposition \ref{lem : number theory} we give a proof of the main result of our paper.

\begin{proof}[Proof of Theorem \ref{th : main 2.5}]
From Theorem~\ref{th : CS/TM}, the claim of Theorem \ref{th : main 2.5} holds if we are able to prove that there is no $\Wk_k$-equivariant map $Z \longrightarrow S(W_k\oplus V^{\oplus j})$,
whose restriction on $Z^{>1}$ is $\Wk_k$-homotopic to $\nu\circ\Psi_{\mathcal{M}_0}|_{Z^{>1}}$, where $\mathcal{M}_0$ is a fixed collection of $j$ nice measures on $\R^d$ such that no $k$-element affine hyperplane arrangement parameterized by $Z^{>1}$ bisects them.
Here  $Z := \Wk_k\cdot\theta$, $Z^{>1} := Z \cap X_{d,k}^{>1}$ and $\theta:=D_{1+\ell,\ldots,1+\ell,1}^{+,\ldots,+,+}(1,2,\ldots,k)$.

An $\Wk_k$-equivariant map $Z \longrightarrow S(W_k\oplus V^{\oplus j})$ with described properties, according to Proposition  \ref{th : main 3}, does not exist if either,
\begin{compactitem}[\quad---]
\item  $dk=j$ and $\frac{1}{k!}{ dk \choose {d,\dots,d}}$ is odd, or
\item there exists an integer $\ell$ such that  $1 \leq \ell \leq  d-1$,  $(d-\ell)k+\ell=j$ and ${(d-\ell)k+\ell \choose d}\tfrac{1}{(k-1)!}{(d-\ell)(k-1) \choose {d-\ell,\dots,d-\ell}}$ is odd.
\end{compactitem}
Next, Proposition \ref{lem : number theory} implies that when:
\begin{compactitem}[\quad---]
\item $d=2^a$ for some integer $a\geq 0$ and $dk=j$, the number $\frac{1}{k!}{ dk \choose {d,\dots,d}}$ is odd, and for
\item $(d-\ell)k+\ell=j$,  $2\leq 2 \ell \leq d-1$, $d=2^a+\ell$ for some integer $a\geq 1$, the number ${(d-\ell)k+\ell \choose d}\cdot\frac{1}{(k-1)!}{(d-\ell)(k-1) \choose {d-\ell,\dots,d-\ell}}$ is odd.
\end{compactitem}
Thus we concluded the proof of Theorem \ref{th : main 2.5}.
\end{proof}

\medskip
In the next two parts of this section we verify both ingredients of the proof of Theorem \ref{th : main 2.5}, that is, we prove Proposition \ref{th : main 3} and Proposition \ref{lem : number theory}.

%
%
\subsection{Proof of Proposition \ref{th : main 3} and Corollary \ref{cor : existence if and only if}}
In order to answer the question about the existence of the equivariant map \eqref{eq : eq-map-2} we use the relative equivariant obstruction theory of tom Dieck on an $\Wk_k$-invariant subcomplex $Z$ of the sphere $X_{d,k}$ with respect to the group of signed permutations $\Wk_k$.
For that we follow \cite[Sec.\,2.6 and Sec.\,4]{BlagojevicFrickHaaseZiegler2} and use the $\Wk_k$-CW structure on $(X_{d,k},X_{d,k}^{>1})$ introduced in \cite[Sec.\,3]{BlagojevicFrickHaaseZiegler2} and presented in Section \ref{sec : appendix}.
A concise presentation of the relevant equivariant obstruction theory can be found in \cite[Sec.\,II.3]{tom1987transformation}.

\medskip
The study of the existence of the equivariant map \eqref{eq : eq-map-2} is done in three separate steps.
First, in Section \ref{subsec : step 1}, we check all the relevant assumptions needed for an application of  relative obstruction theory.
Furthermore, we identify what is the first obstruction which needs to be calculated and what is the ambient group where this obstruction lives.
In the second step, Section \ref{subsec : step 2}, we explain how the obstruction cocycle will be computed using the the binomial moment curve \eqref{eq : moment curve} and give the formula \eqref{eq:obstruction_cocycle} for the evaluation of the cocycle on a cell of a corresponding $\Wk_k$-CW complex.
The third step, proof of the (non-)vanishing of the cohomology class of the obstruction cocycle, is presented in Section \ref{subsec : step 3-1} for the case when the primary obstruction is the only obstruction for the existence of the equivariant map \eqref{eq : eq-map-2}, and in Section \ref{subsec : step 3-2} for the case when there are more obstructions.

\subsubsection{Setting up the obstruction theory}
\label{subsec : step 1}
We consider the problem of the existence of an $\Wk_k$-equivariant map
\begin{equation}
	\label{eq-map-2}
	 Z \longrightarrow  S(W_k\oplus V^{\oplus j}),
\end{equation}
whose restriction to the subcomplex $Z^{>1} := Z \cap X_{d,k}^{>1}$ is $\Wk_k$-homotopic to the map $\nu\circ\Psi_{\mathcal{M}_0}|_{Z^{>1}}$, where $\mathcal{M}_0$ is a some fixed collection of $j$ nice measures on $\R^d$ such that no $k$-element affine hyperplane arrangement parameterized by $Z^{>1}$ bisects them.

\medskip
Let us denote the dimensions of $Z$ and of the sphere $S(W_k\oplus V^{\oplus j})$ as follows
\[
M:= \dim Z
\qquad
\text{and}
\qquad
N:=\dim (S(W_k\oplus V^{\oplus j}))=j+k-2.
\]
In the case \ref{th : main 3}\eqref{th : main 3 : case 1} we have
\begin{multline*}
M	:= \dim Z = \dim X_{d,k} =   \dim D_{1,\ldots,1,1}^{+,\ldots,+,+}(1,2,\ldots,k) =\\(d+1)k - 1 =  j + k -1  = N + 1
\end{multline*}
because  $j = dk$.
In the case \ref{th : main 3}\eqref{th : main 3 : case 2} we have $j = (d-\ell)k + \ell$ and consequently
\begin{multline*}
	M  :=\dim Z = \dim D_{1+\ell,\ldots,1+\ell,1}^{+,\ldots,+}(1,2,\ldots,k)  = \\ (d+1)k -1 - \ell(k-1)  = j + k -1
	   \ = N  + 1.
\end{multline*}

\medskip
In order to apply relative equivariant obstruction theory, as presented by tom Dieck in \cite[Sec.\,II.3]{tom1987transformation}, the following requirements need to be satisfied:
\begin{compactitem}[---]

\item $Z$ is equipped with the structure of a relative $\Wk_k$-CW complex $(Z,Z^{>1})$.
This is obtained from the relative $\Wk_k$-CW structure of $(X_{d,k},X_{d,k}^{>1})$, as demonstrated in \cite[Sec.\,3]{BlagojevicFrickHaaseZiegler2}; see Section \ref{sec : appendix}.

\item The $N$-sphere $S(W_k\oplus V^{\oplus j})$ is path connected and $N$-simple; for a definition consult for example \cite[Def.\,5.5.7]{arkowitz2011}.
Indeed, we have that $N\geq 1$, and consequently $\pi_N(S(W_k\oplus V^{\oplus j}))\cong\Z$ is abelian for $N=1$, while $\pi_N(S(W_k\oplus V^{\oplus j}))=0$ when $N\geq 2$.

\item The collection of nice measures $\mathcal{M}_0$ induces the $\Wk_k$-equivariant map
\[
h\colon Z^{>1}\longrightarrow S(W_k\oplus V^{\oplus j}), \qquad h:=\nu\circ\Psi_{\mathcal{M}_0}|_{Z^{>1}},
\]
which we want to extend.
\end{compactitem}

\medskip
The $N$-sphere $S(W_k\oplus V^{\oplus j})$ is $(N-1)$-connected.
Hence, the fixed $\Wk_k$-equivariant map $h\colon Z^{>1}\longrightarrow S(W_k\oplus V^{\oplus j})$ can be extended to an $\Wk_k$-equivariant map
\[
g\colon \sk_N(Z)\cup Z^{>1} \longrightarrow S(W_k\oplus V^{\oplus j}),
\]
where $\sk_N(Z)$ denotes the $N$th skeleton of $Z$.
Since we have
that $M= N + 1$, we now try to extend the map $g$ to the next, final, $(N+1)$th skeleton of $Z$.
The extension of the map $g$ is obstructed by the equivariant cocycle
\[
\oo(g)\in \mathcal{C}_{\Wk_k}^{N+1}\big(Z,Z^{>1} \,; \,\pi_{N}(S(W_k\oplus V^{\oplus j}))\big),
\]
while the extension of the map $g|_{\sk_{N-1}(Z)\cup Z^{>1}}$ is obstructed by the  cohomology class
\[
[\oo(g)]\in \HH_{\Wk_k}^{N+1}\big(Z,Z^{>1} \,; \,\pi_{N}(S(W_k\oplus V^{\oplus j}))\big).
\]
The cocycle $\oo(g)$ and the cohomology class $[\oo(g)]$ are called the obstruction cocycle and respectively the obstruction element  associated to the map $g$.
Now, the central theorem \cite[Thm.\,II.3.10]{tom1987transformation} tells us that:
\begin{compactitem}[\quad ---]

\item The $\Wk_k$-equivariant map $g\colon \sk_N(Z)\cup Z^{>1} \longrightarrow S(W_k\oplus V^{\oplus j})$ extends to the next skeleton  $\sk_{N+1}(Z)\cup Z^{>1} = Z$ if and only if the obstruction cocycle vanishes, that is $\oo(g)=0$.

\item The restriction $\Wk_k$-equivariant map
\[g|_{\sk_{N-1}(Z)\cup Z^{>1}}\colon \sk_{N-1}(Z)\cup Z^{>1} \longrightarrow S(W_k\oplus V^{\oplus j})\]
extends to $\sk_{N+1}(Z)\cup Z^{>1} = Z$ if and only if the obstruction element vanishes, that is $[\oo(g)]=0$.

\end{compactitem}
Furthermore, since
\[
\dim ( \sk_N(Z)\cup Z^{>1})-\dim (S(W_k\oplus V^{\oplus j}))=1
\]
and
\[
\conn(S(W_k\oplus V^{\oplus j}))=N-1
\]
according to \cite[Prop.\,II.3.15]{tom1987transformation} any two $\Wk_k$-equivariant maps
\[
g',g''\colon \sk_N(Z)\cup Z^{>1} \longrightarrow S(W_k\oplus V^{\oplus j})
\]
define cohomologous obstruction cocycles $\oo(g')$ and $\oo(g'')$, or in other words the induced obstruction elements coincide: $[\oo(g')]=[\oo(g'')]$.
Thus, {\em it is enough to compute the obstruction element $[\oo(\nu\circ\Psi_{\mathcal{M}_0}|_{\sk_N(Z)\cup Z^{>1}})]$ associated to the map $\nu\circ\Psi_{\mathcal{M}_0}|_{\sk_N(Z)\cup Z^{>1}}$ induced by a fixed collection of $j$ nice measures $\mathcal{M}_0$, which have the property that no $k$-element affine hyperplane arrangement parameterized by the subcomplex $Z^{>1}$ bisects them, or in other words for which $0\notin \im (\Psi_{\mathcal{M}_0} |_{\sk_N(Z)})$}.

\subsubsection{Evaluation of the obstruction cocycle $\oo(\nu\circ\Psi_{\mathcal{M}_0}|_{\sk_N(Z)\cup Z^{>1}})$}
\label{subsec : step 2}
With the fixed cellular structure we assume that an orientation on each cell of the $\Wk_k$-CW complex $Z$ is chosen.
Furthermore, we choose an orientation on the sphere $S(W_k\oplus V^{\oplus j})$.

\medskip
Let $\theta$ be an arbitrary $(N+1)$-dimensional cell of $Z$, $f_{\theta}\colon E^{N+1}\longrightarrow Z$ be the associated characteristic map, and let $e_{\theta}$ denote the corresponding basis element in the cellular chain group $C_{N+1}(Z,Z^{>1})$.
Here $E^{N+1}$ denotes the $(N+1)$-dimensional ball.
Then by the geometric definition of the obstruction cocycle associated to the map $\nu\circ\Psi_{\mathcal{M}_0}|_{\sk_N(Z)\cup Z^{>1}}$ we have that
\[
\oo(\nu\circ\Psi_{\mathcal{M}_0}|_{\sk_N(Z)\cup Z^{>1}}) (e_{\theta}) = [\nu\circ\Psi_{\mathcal{M}_0}\circ f_{\theta}|_{\partial\theta}]\in \pi_{N}(S(W_k\oplus V^{\oplus j})).
\]
For more details of the geometric definition of the obstruction cocycle consult for example \cite[Sec.\,7.3]{davis2001}.
The spheres $\partial\theta$ and $S(W_k\oplus V^{\oplus j})$ have the same dimension and therefore the homotopy class $[\nu\circ\Psi_{\mathcal{M}_0}\circ f_{\theta}|_{\partial\theta}]$ is completely determined by the degree of the map
\[
\xymatrix@1{
 \partial\theta\ar[rr]^-{f_{\theta}|_{\partial\theta}} \  & & \  \sk_N(Z)\cup Z^{>1}\ar[rrr]^{\nu\circ\Psi_{\mathcal{M}_0}|_{\sk_N(Z)\cup Z^{>1}}} \  & & & \ S(W_k\oplus V^{\oplus j}).
}
\]
Recall that the orientations on $\partial\theta$ and $S(W_k\oplus V^{\oplus j})$ are already fixed and so the degree is well defined.
For simplicity, let $\kappa:= \nu\circ\Psi_{\mathcal{M}_0}|_{\sk_N(Z)\cup Z^{>1}}\circ f_{\theta}|_{\partial\theta}$.

\medskip
Now we want to evaluate degree of the map $\kappa\colon \partial\theta\longrightarrow  S(W_k\oplus V^{\oplus j})$.
For that we fix $\mathcal{M}_0$ to be the collection of nice measures $(\mu_1,\ldots,\mu_j)$ where $\mu_{r}$ is the measure concentrated on the segment $I_r:=\gamma([t_{r1},t_{r2}])$ of the binomial moment curve in~$\R^d$
\begin{equation}
	\label{eq : moment curve}
	\gamma(t)=\big(\tbinom{t}{1},\tbinom{t}{2},\tbinom{t}{3},\ldots, \tbinom{t}{d} \big)^T,
\end{equation}
where
\[
\ell< t_{11}<t_{12}<t_{21}<t_{22}<\cdots<t_{j1}<t_{j2}.
\]
(Here $A^T$ stands for the transposition of the matrix $A$.)  
In the case $Z = X_{d,k}$ we take $\ell = 0$.
The intervals $(I_1,\ldots,I_j)$ determined by $t_{r1}<t_{r2}$ can be chosen in such a way that $0\notin \im (\Psi_{\mathcal{M}_0}|_{\sk_N(Z)\cup Z^{>1}})$.
This requirement will be directly verified for every concrete situation in the next section.

\medskip
The binomial moment curve is used because the cell~$\theta=D^{+,+,+,\ldots,+}_{\ell+1,1,1,\ldots,1}(1,2,3,\ldots,k)$ parametrizes all arrangements $\mathcal{H}=(H_1,\ldots,H_k)$ of $k$ linear hyperplanes in $\R^{d+1}$, where the order and orientation are fixed appropriately, such that
\begin{compactitem}[\quad--- ]
\item $\{(1,\gamma(0)),\ldots, (1,\gamma(\ell-1))\}\subseteq H_1$,
\item $(1,\gamma(\ell))\notin H_1$, 
\item $(1,\gamma(0))\notin H_2,\ldots, (1,\gamma(0))\notin H_k$, and
\item  $H_2,\dots,H_k$ have unit normal vectors  $x_2,\dots,x_k$ with distinct (positive) first coordinates, that is,
       \[
       \Big|\big\{ \langle x_2, (1,\gamma(0))\rangle, \langle x_3,(1,\gamma(0))\rangle,\ldots,\langle x_k,(1,\gamma(0))\rangle  \big\}\Big|=k-1.
       \]
\end{compactitem}
For the complete account of these facts see \cite[Sec.\,3.4]{BlagojevicFrickHaaseZiegler2}.

\medskip
Next, consider the commutative diagram:
\begin{equation*}
\xymatrix{
\partial\theta\ar[r]_-{f_{\theta}|_{\partial\theta}}\ar[d]\ar@/^1.5pc/[rrrrr]^-{\kappa} & \sk_N(Z)\cup Z^{>1}\ar[rrr]_{\Psi_{\mathcal{M}_0}|_{\sk_N(Z)\cup Z^{>1}}} \ar[d] & & & (W_k\oplus V^{\oplus j}){\setminus}\{0\} \ar[d]\ar[r]_{\nu} & S(W_k\oplus V^{\oplus j})\\
\theta\ar[r]^-{f_{\theta}}\ar@/_1.5pc/[rrrr]^-{\widehat{\kappa}}  & Z\ar[rrr]^{\Psi_{\mathcal{M}_0}|_{Z}}  & & & W_k\oplus V^{\oplus j}.
}
\end{equation*}
Here the vertical arrows are inclusions, and the composition of the lower horizontal maps is denoted by $\widehat{\kappa}:=\Psi_{\mathcal{M}_0}|_{Z}\circ f_{\theta}$.
Now, let $E_{\varepsilon}(0)$ denote the ball with center $0$ in the $\Wk_k$-representation $W_k\oplus V^{\oplus j}$ of, a sufficiently small, radius~$\varepsilon>0$.
Furthermore, let $\widetilde{\theta}:=\theta{\setminus}\widehat{\kappa}^{-1}(E_{\varepsilon}(0))$.
Because of the equality of dimensions $\dim (\theta)=\dim (W_k\oplus V^{\oplus j})$ we can assume that the set of zeros $\widehat{\kappa}^{-1}(0)\subseteq\relint(\theta)$ is finite, say of cardinality $z\geq0$.
Again finiteness of set of zeroes of the function $\widehat{\kappa}$ is checked in every concrete case independently.

The function $\widehat{\kappa}$ is a restriction of the function $\Psi_{\mathcal{M}_0}$ and therefore {\em the points in $\widehat{\kappa}^{-1}(0)$ correspond to the $k$-element affine hyperplane arrangements in $\relint\theta$ which bisect $\mathcal{M}_0$}.
From the fact that:
\begin{compactitem}[\rm \quad ---]
\item the measures in $\mathcal{M}_0$ are disjoint intervals on a moment curve \eqref{eq : moment curve}, and that
\item each hyperplane  cuts the moment curve in at most $d$ distinct points,
\end{compactitem}
it follows that each zero in $\widehat{\kappa}^{-1}(0)$ is isolated and transversal.
The boundary of $\widetilde{\theta}$ is composed of the boundary of the cell $\partial\theta$ and in addition $z$ disjoint copies of $N$-spheres $S_1,\ldots,S_z$, one for each zero of $\widehat{\kappa}$, which are contained in the relative interior of the cell $\theta$.
Therefore, the fundamental class of the sphere $\partial\theta$ is equal to the sum (up to a sign) of fundamental classes $\sum [S_i]$ in $H_{N}(\widetilde{\theta};\Z)$.
Keep in mind that the fundamental class of $\partial\theta$ is determined by the cell orientation inherited from the $\Wk_k$-CW structure on $Z$, which we already fixed.
Now we define orientation on the spheres $S_1,\ldots,S_z$ in such a way that equality $[\partial\theta]=\sum [S_i]$ is valid.
Consequently,
\[
\sum (\nu\circ\widehat{\kappa}|_{\widetilde\theta})_{*}([S_i]) = (\nu\circ\widehat{\kappa}|_{\widetilde\theta})_{*}([\partial\theta])=\kappa_{*}([\partial\theta])=\deg(\kappa)\cdot[S(W_k\oplus V^{\oplus j})].
\]
Rearranging the left hand side of the equality using the family of continuous maps $\nu\circ\widehat{\kappa}|_{S_i}\colon S_i\longrightarrow S(W_k\oplus V^{\oplus j})$ we get that
\[
\sum (\nu\circ\widehat{\kappa}|_{\widetilde\theta})_{*}([S_i])=
\sum (\nu\circ\widehat{\kappa}|_{S_i})_{*}([S_i])=
\Big(\sum\deg(\nu\circ\widehat{\kappa}|_{S_i})\Big)\cdot[S(W_k\oplus V^{\oplus j})].
\]
Hence,
\[
\deg(\kappa) = \sum\deg(\nu\circ\widehat{\kappa}|_{S_i}).
\]
where the sum ranges over all $k$-element affine hyperplane arrangements in $\relint(\theta)$ which bisect $\mathcal{M}_0$.
Thus we have obtained that
\begin{align}
\label{eq:obstruction_cocycle}
\oo(\nu\circ\Psi_{\mathcal{M}_0}|_{\sk_{N}(Z)\cup Z^{>1}})(e_{\theta} ) &= [\nu\circ\Psi_{\mathcal{M}_0}\circ f_{\theta}|_{\partial\theta}]\\
&=[\kappa]\nonumber \\
&=\deg(\kappa)\cdot\zeta\nonumber \\
&= \sum\deg(\nu\circ\widehat{\kappa}|_{S_i})\cdot\zeta.\nonumber
\end{align}
Here $\zeta\in \pi_{N}(S(W_k\oplus V^{\oplus j}))\cong H_{N}(S(W_k\oplus V^{\oplus j});\Z)\cong\Z$ is the generator determined by the already fixed orientation on the sphere.
The sum \eqref{eq:obstruction_cocycle} ranges over all $k$-element affine hyperplane arrangements in $\relint(\theta)$ that bisect $\mathcal{M}_0$.

\subsubsection{Evaluation of the obstruction element in the case $Z = X_{d,k}$}
\label{subsec : step 3-1}
In this section we complete the proof of Proposition \ref{th : main 3}\eqref{th : main 3 : case 1} and Corollary \ref{cor : existence if and only if}.

\medskip
Recall that $M = N+1$ and thus
$[\oo(\nu\circ\Psi_{\mathcal{M}_0}|_{\sk_N(X_{d,k})\cup X_{d,k}^{>1}})]$
is the primary obstruction element and also the only obstruction for the existence of the map \eqref{eq-map-2}.
In particular, this means that an
$\Wk_k$-equivariant map $X_{d,k} \longrightarrow  S(W_k\oplus V^{\oplus j})$,
whose restriction on $X_{d,k}^{>1}$ is $\Wk_k$-homotopic to
	$\nu\circ\Psi_{\mathcal{M}_0}|_{X_{d,k}^{>1}}$, exists if and only if
$[\oo(\nu\circ\Psi_{\mathcal{M}_0}|_{\sk_N(X_{d,k})\cup X_{d,k}^{>1}})]=0$.
We will prove that
\begin{equation}
	\label{eq : obstruction-01}
	[\oo(\nu\circ\Psi_{\mathcal{M}_0}|_{\sk_N(X_{d,k})\cup X_{d,k}^{>1}})]=0
	\qquad\Longleftrightarrow\qquad
	 \frac{1}{k!}{ dk \choose {d,\dots,d}}
	 \text{ is even.}
\end{equation}
This would conclude the proof of Proposition \ref{th : main 3}\eqref{th : main 3 : case 1} and Corollary \ref{cor : existence if and only if}.

\medskip
We have to evaluate the cocycle
\[
\oo:=\oo(\nu\circ\Psi_{\mathcal{M}_0}|_{\sk_N(X_{d,k})_\cup X_{d,k}^{>1}})\in \mathcal{C}_{\Wk_k}^{N+1}\big(X_{d,k},X_{d,k}^{>1} \,; \,\pi_{N}(S(W_k\oplus V^{\oplus j}))\big),
\]
on the $M(=N+1)$-cells of the $M$-dimensional sphere $X_{d,k}$.
From \cite[Thm.\,3.11]{BlagojevicFrickHaaseZiegler2} we know that $X_{d,k}$ has a unique full $\Wk_k$-orbit of maximal dimensional cells represented by the cell
\[
\theta:= D_{1,\ldots,1}^{+,\ldots,+}(1,2,\ldots,k).
\]
Furthermore, from Theorem \ref{th : CW-model} or \cite[Ex.\,3.12]{BlagojevicFrickHaaseZiegler2}, we have that $\theta$ is given by the inequalities $x_{1,1}<x_{1,2}< \dots < x_{1,k}$.
Thus, having in mind that $\oo$ is an $\Wk_k$-equivariant cocycle, it suffices to evaluate $\oo(e_{\theta})$.

\medskip
Consider a collection of $j$ ordered disjoint intervals $\mathcal{M}_0=(I_1, \ldots, I_j)$ along the moment curve $\gamma$, defined in \eqref{eq : moment curve}, with midpoints $(x_1,\ldots,x_j)$ respectively.
Then, according to \eqref{eq:obstruction_cocycle}, we have that
\begin{equation}
	\label{ob-01}
	\oo(e_{\theta}) = \sum\deg(\nu\circ\widehat{\kappa}|_{S_i})\cdot\zeta =\big(\sum\pm 1\big)\cdot\zeta = :a\cdot\zeta,
\end{equation}
where the sum ranges over all $k$-element affine hyperplane arrangements in $\relint(\theta)$ which bisect $\mathcal{M}_0$.
We have that:
\begin{compactitem}[\quad ---]
	\item $dk=j$,
	\item any $k$-element affine hyperplane arrangement in $\R^d$ has at most $dk$ intersection points with the moment curve $\gamma$,
	\item for bisection of a collection of $j$ intervals on $\gamma$ one needs at least $j$ points, and
	\item each $k$-element affine hyperplane arrangement that bisects $\mathcal{M}_0$ is completely determined (up to an orientation of hyperplanes) by a partition of the set of midpoints  $\{x_1,\ldots,x_j\}$ of the intervals $(I_1, \ldots, I_j)$ into $k$ subset of cardinality $d$ each, where each of these subset uniquely determines a hyperplane of the $k$-element affine hyperplane arrangement.
\end{compactitem}
Thus the number of $k$-element affine hyperplane arrangements that bisect $\mathcal{M}_0$ is
${dk \choose {d,\dots,d}}2^k$.
Using slight perturbations of the intervals $(I_1, \ldots, I_j)$ along the curve $\gamma$, without changing their order, we can assume that all the bisecting $k$-element affine hyperplane arrangements are contained in $\bigcup_{g\in\Wk_k}g\cdot\relint(\theta)$.
Thus, the number of $k$-element affine hyperplane arrangements that bisect $\mathcal{M}_0$ and are contained in $\relint(\theta)$ is $\frac{1}{k!}{ dk \choose {d,\dots,d}}$.
This means that the integer $a$, defined by the equation \eqref{ob-01}, has the property that
\[
a \equiv \frac{1}{k!}{ dk \choose {d,\dots,d}} \mod 2.
\]

\medskip
In the final step let us assume that $[\oo]=0$, meaning that the cocycle $\oo$ is also a coboundary.
Thus there exists a cochain
\[
\bb\in \mathcal{C}_{\Wk_k}^{N}\big(X_{d,k},X_{d,k}^{>1} \,; \,\pi_{N}(S(W_k\oplus V^{\oplus j}))\big)
\]
such that $\oo=\delta\bb$, where $\delta$ denotes the coboundary operator.
From \eqref{rel - 1} or \cite[Eq.\,(11)]{BlagojevicFrickHaaseZiegler2} we have that
\begin{equation}
    \label{rel - 1a}
      \partial e_{\theta}=(1+(-1)^d\varepsilon_1)\cdot e_{\gamma_1}+\sum_{i=2}^{k}
      (1+(-1)^d\tau_{i-1,i})\cdot e_{\gamma_{2i-1}},
      \end{equation}
where the cells $\gamma_1,\ldots,\gamma_{2k}$ are described in Example \ref{example:boundary-1}, or  in \cite[p.\,755]{BlagojevicFrickHaaseZiegler2}, and $\tau_{i-1,i}\in\Sym_k\subseteq\Wk_k$ denotes the transposition that interchanges $i-1$ and $i$.
Thus, $\oo=\delta\bb$ and \eqref{rel - 1a} imply that
\begin{align*}
a\cdot \zeta &= \oo(e_{\theta})=\delta\bb(e_{\theta}) =\bb(\partial e_{\theta})\\
&=(1+(-1)^d\varepsilon_1)\cdot \bb(e_{\gamma_1})+\sum_{i=2}^{k}
      (1+(-1)^d\tau_{i-1,i})\cdot \bb(e_{\gamma_{2i-1}}) \\
&=  (1+(-1)^{d+j})\cdot \bb(e_{\gamma_1})+\sum_{i=2}^{k}
      (1+(-1)^{d+1})\cdot \bb(e_{\gamma_{2i-1}})\\
&= 2b\cdot\zeta,
\end{align*}
for some integer $b$.
In this calculation we use the fact that $\bb$ is an equivariant cochain, and that $\varepsilon_1$ and $\tau_{i-1,i}$ act on $V^{\oplus j}$ respectively by multiplication with $(-1)^j$ and trivially. 
Whereas, $\varepsilon_1$ and $\tau_{i-1,i}$ act on $W_k$ trivially and by multiplication with $(-1)$ respectively.
Hence,
\[
[\oo]=0
\quad\Longleftrightarrow\quad
a\equiv 0\mod 2
\quad\Longleftrightarrow\quad
\frac{1}{k!}{ dk \choose {d,\dots,d}}\equiv  0 \mod 2.
\]
We verified \eqref{eq : obstruction-01}, and concluded a proof of Proposition \ref{th : main 3}\eqref{th : main 3 : case 1} and Corollary \ref{cor : existence if and only if}.

\subsubsection{Evaluation of the obstruction element in the case $Z = \Wk_k\cdot\overline{\theta}$ where $\theta=D_{1+\ell,\ldots,1+\ell,1}^{+,\ldots,+,+}(1,2,\ldots,k)$}
\label{subsec : step 3-2}
In this section we complete the proof of  Proposition \ref{th : main 3}\eqref{th : main 3 : case 2}.

\medskip
As before we have that $\dim Z=M = N+1$ and consequently the obstruction element
$[\oo(\nu\circ\Psi_{\mathcal{M}_0}|_{\sk_N(Z)\cup Z^{>1}})]$
is the primary obstruction element and the only obstruction to the existence of an $\Wk_k$-equivariant map \eqref{eq-map-2}.
However, in this case, it is not the only obstruction for the existence of an $\Wk_k$-equivariant map $X_{d,k}\longrightarrow S(W_k\oplus V^{\oplus j})$.

\medskip
Thus, we prove the implication
\begin{multline}
\label{eq : obstruction-012}
{(d-\ell)k+\ell \choose d}\cdot\frac{1}{(k-1)!}{(d-\ell)(k-1) \choose {d-\ell,\dots,d-\ell}} \equiv 1 \mod 2  \\
\Longrightarrow \quad
[\oo(\nu\circ\Psi_{\mathcal{M}_0}|_{\sk_{N}(Z)\cup Z^{>1}})]\neq 0.
\end{multline}
In this way we would prove Proposition \ref{th : main 3}\eqref{th : main 3 : case 2} and complete the proof of the proposition.

\medskip
For that we evaluate the obstruction cocycle
\[
\oo:=\oo(\nu\circ\Psi_{\mathcal{M}_0}|_{\sk_{N}(Z)\cup Z^{>1}})\in \mathcal{C}_{\Wk_k}^{N+1}\big(Z,Z^{>1} \,; \,\pi_{N}(S(W_k\oplus V^{\oplus j}))\big).
\]
on the $M(=N+1)$-cells of $Z$.
By construction $Z$ is given as the $\Wk_k$-orbit of the cell
\[
    \theta:=D_{1+\ell,\ldots,1+\ell,1}^{+,\ldots,+,+}(1,2,\ldots,k).
\]
and the boundary of its generator $e_{\theta}$ can be written as
\begin{align}
\label{eq : boundary}
\partial e_{\theta}&=(1+(-1)^{d-1}\varepsilon_1)e_{\nu_1}+\sum_{i=2}^{k-1}(1+(-1)^{d-1}\tau_{i-1,i})e_{\nu_{2i-1}}+ \\
&\qquad\qquad\qquad\qquad\qquad\qquad\qquad\qquad\qquad\qquad\qquad \sum_{i=w}^{k} (1+ (-1)^{d} \varepsilon_{i})e_{\mu_{2i}}, \nonumber
\end{align}
where $w={\small \begin{cases}1, & \ell = 1\\ k, & \ell \neq 1\end{cases}}$.
For more details about the cell $\theta$ consult Example \ref{ex:cell_example} or \cite[pp.\,751,\,754]{BlagojevicFrickHaaseZiegler2}.

\medskip
Now we will evaluate $\oo(\theta)$.
Since $\oo$ is an $\Wk_k$-equivariant cocycle, in this way we will evaluate the cocycle $\oo$ on all the cells in the orbit of $\theta$.

\medskip
Consider the moment curve $\gamma$ defined in \eqref{eq : moment curve}.
We fix a collection of $j$ ordered disjoint intervals $\mathcal{M}_0=(I_1, \ldots, I_j)$ on $\gamma$ defined by $I_1=\gamma([t_{11},t_{12}]),\dots, I_j=\gamma([t_{j1},t_{j2}])$ where
\[
\ell< t_{11}<t_{12}<t_{21}<t_{22}<\cdots<t_{j1}<t_{j2}.
\]
Then, as in \cite[Lem.\,3.13]{BlagojevicFrickHaaseZiegler2}, we have that the cell $\theta$ parameterizes all $k$-element affine hyperplane arrangements, where the order and orientation are fixed appropriately, such that the first $k-1$ hyperplanes contain the points $s_1:=\gamma(0),s_2:=\gamma(1),\dots,s_\ell:=\gamma(\ell-1)$.
Thus again, according to \eqref{eq:obstruction_cocycle}, we have that
\begin{equation}
	\label{ob-02}
	\oo(e_{\theta}) = \sum\deg(\nu\circ\widehat{\kappa}|_{S_i})\cdot\zeta =\big(\sum\pm 1\big)\cdot\zeta = :a\cdot\zeta,
\end{equation}
where the sum ranges over all $k$-element affine hyperplane arrangements in $\relint(\theta)$ which bisect $\mathcal{M}_0$.
We have that:
\begin{compactitem}[\quad ---]
	\item $(d-\ell)k+\ell=j$,
	\item any $k$-element affine hyperplane arrangement in $\R^d$ has at most $dk$ intersection points with the moment curve $\gamma$,
	\item $\theta$ parameterizes all $k$-element affine hyperplane arrangements such that first $k-1$ hyperplane contain the points $s_1,\dots,s_\ell$, meaning that $(k-1)\ell$ intersection points out of $dk$ cannot be used for interval partitioning,
	\item for bisection of collection of $j$ intervals on $\gamma$ one needs at least $j=dk-\ell(k-1)$ points, and thus
	\item each $k$-element affine hyperplane arrangement from $\theta$ that bisects $\mathcal{M}_0$ is completely determined (up to an orientation of hyperplanes) by a partition of the set of midpoints  $\{x_1,\ldots,x_j\}$ of the intervals $(I_1, \ldots, I_j)$ into $k-1$ subset of cardinality $d-1$ each and one subset of cardinality $D$, where each of these subsets uniquely determines a hyperplane of the $k$-element affine hyperplane arrangement.
\end{compactitem}
Consequently, the number of $k$-element affine hyperplane arrangements from $Z$ which bisect $\mathcal{M}_0$ is
${(d-\ell)k+\ell \choose d}\cdot\frac{1}{(k-1)!}{(d-\ell)(k-1) \choose {d-\ell,\dots,d-\ell}}\cdot  2^{k-1}$.
Again, using slight perturbations of the intervals $(I_1, \ldots, I_j)$ along the curve $\gamma$, without changing their order, we can assume that all the bisecting $k$-element affine hyperplane arrangements are contained in $\bigcup_{g\in\Wk_k}g\cdot\relint(\theta)$.
Thus, the number of $k$-element affine hyperplane arrangements that bisect $\mathcal{M}_0$ and are contained in $\relint(\theta)$ is ${(d-\ell)k+\ell \choose d}\cdot\frac{1}{(k-1)!}{(d-\ell)(k-1) \choose {d-\ell,\dots,d-\ell}}$.
This means that the integer $a$, defined by equation \eqref{ob-02}, has the property
\[
a \equiv {(d-\ell)k+\ell \choose d}\cdot\frac{1}{(k-1)!}{(d-\ell)(k-1) \choose {d-\ell,\dots,d-\ell}} \mod 2.
\]
\medskip
Next, assume that $[\oo]=0$, i.e., the cocycle $\oo$ is also a coboundary.
Hence there is an $N$-cochain
$
\bb\in \mathcal{C}_{\Wk_k}^{N}\big(Z,Z^{>1} \,; \,\pi_{N}(S(W_k\oplus V^{\oplus j}))\big)
$
such that $\oo=\delta\bb$, where $\delta$, as before, is the coboundary operator.
Consequently, \eqref{eq : boundary} implies that
\begin{align*}
a\cdot \zeta &= \oo(e_{\theta})=\delta\bb(e_{\theta}) =\bb(\partial e_{\theta})\\
&=(1+(-1)^{d-1}\varepsilon_1)\cdot \bb(e_{\nu_1})+ \sum_{i=2}^{k-1}
      (1+(-1)^{d-1}\tau_{i-1,i})\cdot \bb(e_{\nu_{2i-1}}) +\\
& \qquad \qquad\qquad\qquad\qquad\qquad\qquad\qquad\qquad\qquad\qquad  \sum_{i=w}^{k} (1+ (-1)^{d} \varepsilon_{i})\bb(e_{\mu_{2i}})\\
&=  (1+(-1)^{d-1+j})\cdot \bb(e_{\nu_1})+\sum_{i=2}^{k}
      (1+(-1)^{d})\cdot \bb(e_{\nu_{2i-1}})+\\
& \qquad \qquad\qquad\qquad\qquad\qquad\qquad\qquad\qquad\qquad\qquad      \sum_{i=w}^{k} (1+ (-1)^{d-j})\bb(e_{\mu_{2i}})\\
&= 2b\cdot\zeta,
\end{align*}
for some integer $b$, where $w={\small \begin{cases}1, & \ell = 1\\ k, &  \ell \neq 1 \end{cases}}$.
Here we use the fact that $\bb$ is an equivariant cochain, and that $\varepsilon_1$ and permutations $\tau_{i-1,i}$ act on $V^{\oplus j}$
respectively by multiplication with $(-1)^j$ and trivially. 
They act on $W_k$ trivially and by multiplication with $(-1)$ respectively.
Therefore, if
\[
a\equiv {(d-\ell)k+\ell \choose d}\cdot\frac{1}{(k-1)!}{(d-\ell)(k-1) \choose {d-\ell,\dots,d-\ell}} \not\equiv 0\mod 2,
\]
then $[\oo]\neq 0$, and we concluded the proof of \eqref{eq : obstruction-012} and
Proposition \ref{th : main 3}\eqref{th : main 3 : case 2}.
%
\subsection{Proof of Proposition \ref{lem : number theory}}\label{sec : proof of lemma}

For the proof of the proposition we use the following classical facts going back to Legendre \cite{Legendre1899}, for a modern reference see for example \cite[Thm.\,2.6.4]{Moll2012}.
Let $p$ be a prime, $k\geq 1$ be an integer, and let $E_p(k):=\max \{i\in\N\cup\{0\} : p^i \mid k!\}$.
There is a unique $p$-adic presentation of the integer $k$ in the form $k=a_0+a_1p+\dots+a_mp^m$, where $0\leq a_i\leq p-1$ for all $0\leq i\leq m$.
Let $\alpha_p(k):=a_0+a_1+\dots+a_m$ denote the sum of coefficients in the $p$-adic expansion of $k$.
Then
\begin{equation}\label{eq : Ep}
E_p(k)=\sum_{j\geq 1}\Big\lfloor\frac{k}{p^j}\Big\rfloor = \frac{k-\alpha_p(k)}{p-1}.
\end{equation}
Furthermore, if $k_1,\dots,k_t$ are non-negative integers such that $k=k_1+\dots+k_t$, then
\begin{equation}\label{eq : divisibility}
{k \choose {k_1,\dots,k_t}}\equiv 0 \mod p^{r}
\qquad\Longleftrightarrow\qquad
E_p(k)-\sum_{i=1}^{t}E_p(k_i)\geq r.
\end{equation}

\medskip
\noindent{\bf (i)}
In our proof we assume that $p=2$, and we also use the inequalities
\begin{equation}\label{eq : alpha inequality}
	\alpha_2(a+b)\leq \alpha_2(a)+\alpha(b)
	\qquad\text{and}\qquad
	\alpha_2(ab)\leq \alpha_2(a)\alpha_2(b),
\end{equation}
whcih hold for arbitrary integers $a\geq 1$ and $b\geq 1$.
Consider the following sequence of equivalences
\begin{eqnarray*}
	\frac{1}{k!}{ dk \choose {d,\dots,d}}=\frac{(dk)!}{k!d!\cdots d!} \qquad\text{is odd}
	&\Longleftrightarrow &
	E_2(dk)=E_2(k)+kE_2(d)\\
	&\Longleftrightarrow &
	E_2(k) = kE_2(d) - E_2(dk)\\
	&\overset{\eqref{eq : Ep}}{\Longleftrightarrow} &
	k-\alpha_2(k)=k\alpha_2(d)-\alpha_2(dk).
\end{eqnarray*}

\medskip
Now, if we assume that $\frac{1}{k!}{ dk \choose {d,\dots,d}}$ is odd, then according of the previous equivalences and \eqref{eq : alpha inequality} we have that
\[
k-\alpha_2(k)=k\alpha_2(d)-\alpha_2(dk)\geq k\alpha_2(d)-\alpha_2(d)\alpha_2(k) = (k-\alpha_2(k))\alpha_2(d).
\]
Since $k\geq 2$ we have that $k-\alpha_2(k)\geq 0$ and consequently $\alpha_2(d)\leq 1$.
Thus $d$ must be a power of two.

\smallskip
On the other hand, let us assume that $d$ is a power of two, or in other words $\alpha_2(d)=1$.
Since in this case $\alpha_2(dk)=\alpha_2(k)$ we get the equality
\[
k-\alpha_2(k)=k\alpha_2(d)-\alpha_2(dk).
\]
Hence, the sequence of equivalences we deduced implies that $\frac{1}{k!}{ dk \choose {d,\dots,d}}$ is odd.

\medskip
\noindent{\bf (ii)}
The product ${(d-\ell)k+\ell \choose d}\cdot \frac{1}{(k-1)!}{(d-\ell)(k-1) \choose {d-\ell,\dots,d-\ell}}$ is odd if and only if both factors are odd.
We know that
\[
\frac{1}{(k-1)!}{(d-\ell)(k-1) \choose {d-\ell,\dots,d-\ell}}\text{ is odd}
\qquad\Longleftrightarrow\qquad
d-\ell=2^a\text{ for some }a\geq 0.
\]
Therefore it remains to discuss when ${(d-\ell)k+\ell \choose d}$ is odd, assuming that $d=2^a+\ell$ where $a\geq 0$ and $2^a > \ell > 0$, which follows from the assumption $d = 2^a + \ell > 2\ell > 0$.
Now, the following sequence of equivalences concludes the proof of the second part of the proposition:
\begin{align*}
	{(d-\ell)k+\ell \choose d} \ \text{is odd} &\Longleftrightarrow
	{2^ak+\ell \choose 2^a+\ell}\ \text{is odd} \\
	&\Longleftrightarrow  E_2(2^ak+\ell)=E_2(2^a+\ell)+E_2(2^a(k-\ell))\\
	&\overset{\eqref{eq : Ep}}{\Longleftrightarrow}
	2^ak+\ell-\alpha_2(2^ak+\ell)= \\
	&\qquad\qquad \quad   2^a+\ell-\alpha_2(2^a+\ell)+2^a(k-1)-\alpha_2(2^a(k-1))\\
	&\Longleftrightarrow
	\alpha_2(2^ak+\ell)=\alpha_2(2^a+\ell)+\alpha_2(2^a(k-1))\\
	&\Longleftrightarrow
	\alpha_2(2^ak+\ell)=\alpha_2(2^a+\ell)+\alpha_2(k-1)\\
	&\overset{2^a > \ell}{\Longleftrightarrow}
		\alpha_2(k)+\alpha_2(\ell)=1+\alpha_2(\ell)+\alpha_2(k-1)\\
		&\Longleftrightarrow  k\text{ is odd.}
\end{align*}

\section{Appendix: A $\Wk_k$-CW structure on the join configuration space}
\label{sec : appendix}

In this section, based on the work in \cite[Sec.\,3]{BlagojevicFrickHaaseZiegler2}, we briefly present a relative $\Wk_k$-CW structure on the join configuration space $X_{d,k}$ which we use in the obstruction theory proof of Theorem \ref{th : main 2.5}.
In particular, this means that the induced $\Wk_k$-CW structure transforms the subspace $X_{d,k}^{>1}$ into an $\Wk_k$-CW subcomplex.
As in the original work, the relative $\Wk_k$-CW complex we construct is denoted by $X:=(X_{d,k},X_{d,k}^{>1})$.
The construction proceeds in two steps:
\begin{compactitem}[\quad---]
\item the Euclidean space $\R^{(d+1)\times k}$ is partitioned into a union of (disjoint) relatively open cones, each containing the origin in its closure, on which the $\Wk_k$-action operates by linearly permuting the cones (Section \ref{sub : Appendix - 01}), and then
\item the open cells of a regular $\Wk_k$-CW model are defined as intersections of these relatively open cones with the unit sphere of the Euclidean space $\R^{(d+1)\times k}$, (Section \ref{sub : Appendix - 02}).
\end{compactitem}

\subsection{A stratification of the Euclidean space $\R^{(d+1)\times k}$}
\label{sub : Appendix - 01}

First we recall the notion of a stratification of a Euclidean space.

\medskip
Let $E$ be a Euclidean space.
A \emph{stratification of $E$ (by cones)} is a finite collection $\C$ of subsets of $E$ that satisfies the following properties:
\begin{compactitem}[\quad ---]
\item $\C$ consists of finitely many non-empty relatively open polyhedral cones of $E$,
\label{def_item_pw_disj_open_pol_cones}
\item $\C$ is a partition of $E$, i.e., $E = \biguplus_{C \in \mathcal{C}} C$,
\label{def_item_R_is_union_of_strata}
\item the closure $\overline C$ of every cone $C \in \C$ is a union of cones in $\C$.
\label{def_item_closure_is_union_of_strata}
\end{compactitem}
An element of the family $\C$ is called a \emph{stratum}.

\medskip
In order to define the desired stratification of the Euclidean space $\R^{(d+1)\times k}$ we first fix following data:
\begin{compactitem}[\quad --- ]
\item a permutation $\sigma:=(\sigma_1,\sigma_2,\ldots,\sigma_k)\in\Sym_k$,
\item a collection of signs $S:=(s_1,\ldots,s_k)\in\{+1,-1\}^k$, and
\item a collection of integers $I:=(i_1,\ldots,i_k)\in\{1,\ldots,d+2\}^k$.
\end{compactitem}
Furthermore, we set $x_0$ to be the origin of the Euclidean space $\R^{(d+1)\times k}$, $\sigma_0=0$ and $s_0=1$.
Now we define the cone
\[
C_I^S(\sigma)=C_{i_1,\ldots,i_k}^{s_1,\ldots,s_k}(\sigma_1,\sigma_2,\ldots,\sigma_k)\subseteq \R^{(d+1)\times k}
\]
to be the collection of all points $(x_1,\ldots,x_k)\in \R^{(d+1)\times k}$, $x_i=(x_{1,i},\ldots,x_{d+1,i})$, such that for each $1\leq t\leq k$,
\begin{compactitem}[\quad --- ]
\item if  $1\leq i_t\leq d+1$, then $s_{t-1}x_{i_t,\sigma_{t-1}}<s_{t}x_{i_t,\sigma_{t}}$ with $s_{t-1}x_{i',\sigma_{t-1}}=s_{t}x_{i',\sigma_{t}}$ for every $i'<i_t$, and
\item if $i_t=d+2$, then $s_{i_{t-1}}x_{\sigma_{t-1}}=s_{i_{t}}x_{\sigma_{t}}$.
\end{compactitem}
A triple $(\sigma|I|S)\in\Sym_k\times \{1,\ldots, d+2\}^k\times \{+1,-1\}^k$ is called a \emph{symbol}.
In the notation of symbols we write instead of the signs $\{+1,-1\}$ just $\{+,-\}$.
The set of ``inequalities'' which define  the stratum $C_I^S(\sigma)$ can be shortly denoted by:
\begin{align*}
C_I^S(\sigma)&=C_{i_1,\ldots,i_k}^{s_1,\ldots,s_k}(\sigma_1,\sigma_2,\ldots,\sigma_k)\\
			 &=\{ (x_1,\ldots,x_k)\in \R^{(d+1)\times k} : 0<_{i_1}s_1x_{\sigma_1}<_{i_2}s_2x_{\sigma_2}<_{i_3}\cdots<_{i_k}s_kx_{\sigma_k}\},
\end{align*}
where $y<_iy'$, for $1\leq i\leq d+1$, means that $y$ and $y'$ agree in the first $i-1$ coordinates and at the $i$-th coordinate $y_i~<~y'_i$.
The inequality $y<_{d+2}y'$ stands for $y=y'$.
Furthermore, each $C_I^S(\sigma)$ equals to the relative interior of a polyhedral cone in $(\R^{d+1})^k$ of codimension $(i_1-1)+\cdots+(i_k-1)$, that means
\[
\dim C_{i_1,\ldots,i_k}^{s_1,\ldots,s_k}(\sigma_1,\sigma_2,\ldots,\sigma_k)=(d+2)k-(i_1+\cdots+i_k).
\]

\medskip
Let $\C$ denote the family of all strata $C_I^S(\sigma)$ defined by all symbols, that is
\[
\C=\big\{C_I^S(\sigma) : (\sigma|I|S)\in \Sym_k\times \{1,\ldots, d+2\}^k\times \{+1,-1\}^k\big\}.
\]
Note that different symbols may define the same sets, additionally:
\[
C_I^S(\sigma)\cap C_{I'}^{S'}(\sigma)\neq\emptyset \ \Longleftrightarrow \ C_I^S(\sigma) = C_{I'}^{S'}(\sigma).
\]
Since it's not hard to check that $\bigcup\C=\R^{(d+1)\times k}$ we conclude that $\C$ is a stratification of the Euclidean space $\R^{(d+1)\times k}$.

\medskip
The action of the group $\Wk_k$ on the Euclidean space $\R^{(d+1)\times k}$ induces an action on the stratification $\C$ by as follows:
\begin{align}
\pi\cdot C^S_I(\sigma) &= C^S_I(\pi\sigma),\label{eq:action of pi}\\
\varepsilon_t\cdot  C^S_I(\sigma) &=\varepsilon_t\cdot C_{i_1,\ldots,i_k}^{s_1,\ldots,s_k}(\sigma_1,\sigma_2,\ldots,\sigma_k)\nonumber\\
                                  &= C_{i_1,\ldots,i_k}^{s_1,\ldots,-s_t,\ldots,s_k}(\sigma_1,\sigma_2,\ldots,\sigma_k),\label{eq:action of epsilon}
\end{align}
where $\pi\in\Sym_k$, $1\leq t\leq k$, and $\varepsilon_1,\ldots,\varepsilon_k$ are the canonical generators of the subgroup $(\Z/2)^k$ of $\Wk_k$.

\subsection{The $\Wk_k$-CW complex induced from the stratification $\C$}
\label{sub : Appendix - 02}

The $\Wk_k$-CW complex structure on the joint configuration space $X_{d,k}=S(\R^{(d+1)\times k})$ is defined by intersecting each stratum $C^S_I(\sigma)$ of the stratification $\C$ with the unit sphere~$S(\R^{(d+1)\times k})$.
Since the stratum $C^S_I(\sigma)$ is a relatively open cone which does not contain a line, the intersection
\[
D^S_I(\sigma)=D_{i_1,\ldots,i_k}^{s_1,\ldots,s_k}(\sigma_1,\sigma_2,\ldots,\sigma_k):=
C_{i_1,\ldots,i_k}^{s_1,\ldots,s_k}(\sigma_1,\sigma_2,\ldots,\sigma_k) \cap S(\R^{(d+1)\times k})
\]
has to be an open cell of dimension $(d+2)k-(i_1+\cdots+i_k)-1$.
The action of the group $\Wk_k$ on the cells $D^S_I(\sigma)$ is induced by from \eqref{eq:action of pi} and \eqref{eq:action of epsilon} as follows:
\begin{align}
\pi\cdot D^S_I(\sigma) &= D^S_I(\pi\sigma),\label{eq:action of pi - 2}\\
\varepsilon_t\cdot  D^S_I(\sigma) &=\varepsilon_t\cdot D_{i_1,\ldots,i_k}^{s_1,\ldots,s_k}(\sigma_1,\sigma_2,\ldots,\sigma_k)\nonumber\\
                                  &= D_{i_1,\ldots,i_k}^{s_1,\ldots,-s_t,\ldots,s_k}(\sigma_1,\sigma_2,\ldots,\sigma_k).\label{eq:action of epsilon - 2}
\end{align}
In this way we have defined a regular $\Wk_k$-CW structure on $X_{d,k}$.
In particular, the action of the group $\Wk_k$ on the Euclidean space $\R^{(d+1)\times k}$ restricts to the cellular action on the model.
Thus, we have the following theorem \cite[Thm.\,3.11]{BlagojevicFrickHaaseZiegler2}.

\begin{theorem}
\label{th : CW-model}
Let $d\geq 1$ and $k\geq 1$ be integers.
The family of cells
\[
\big\{D^S_I(\sigma) : (\sigma|I|S)\neq (\sigma|d+2,\ldots,d+2|S)\big\}
\]
forms a finite regular $((d+1)k-1)$-dimensional $\Wk_k$-CW complex $X:=(X_{d,k},X_{d,k}^{>1})$ which models the join configuration space $X_{d,k}=S(\R^{(d+1)\times k})$.
It has
\begin{compactitem}[ \quad --- ]
\item one full $\Wk_k$-orbit of cells in the maximal dimension $(d+1)k-1$ induced by the cell $D_{1,\ldots,1,1}^{+,\ldots,+,+}(1,2,\ldots,k)$, and
\item $k$ full $\Wk_k$-orbits of cells in dimension $(d+1)k-2$.
\end{compactitem}
The (cellular) $\Wk_k$-action on $X_{d,k}$ is given by relations\eqref{eq:action of pi - 2} and \eqref{eq:action of epsilon - 2}.
Furthermore the collection of cells
\[
\big\{D^S_I(\sigma) : i_s=d+2\text{  for some  }1\leq s\leq k\bigskip \}
\]
is a $\Wk_k$-CW subcomplex and models $X_{d,k}^{>1}$.
\end{theorem}

\medskip
As an illustration of a cell structure we analyze the cells $D_{1,\dots,1}^{+,\ldots,+}(1,2,\dots,k)$ and  $D_{1+\ell,\ldots,1+\ell,1}^{+,\ldots,+,+}(1,2,\ldots,k)$, which are used in the proofs of Proposition \ref{th : main 3}\eqref{th : main 3 : case 1} and \ref{th : main 3}\eqref{th : main 3 : case 2}.
First we recall \cite[Ex.\,3.12]{BlagojevicFrickHaaseZiegler2}.
\begin{example}
\label{example:boundary-1}
Let $d\geq 1$ and $k\geq 1$ be integers.
Consider the cell 
\[
\theta:=D^{+,+,+,\ldots,+}_{1,1,1,\ldots,1}(1,2,3,\ldots,k)
\] 
of the $\Wk_k$-CW complex $X_{d,k}$.
It is determined by the inequalities:
\[
0<_1 x_1<_1 x_2<_1 \cdots <_1 x_k.
\]

\medskip     
The cells of codimension one in the boundary of $\theta$ are obtained by introducing one of the following extra equalities:
\[
      x_{1,1}=0\,,\quad x_{1,1}=x_{1,2}\,,\quad \ldots\quad x_{1,k-1}=x_{1,k}.
\]
Each of these equalities will give two cells, hence there are, in total, $2k$ cells of codimension one in the boundary of the cell $\theta$.
\begin{compactenum}[\rm \quad (a)]
\item The equality $x_{1,1}=0$ induces cells:
\[
      \qquad\qquad\quad \gamma_1:=D^{+,+,+,\ldots,+}_{2,1,1,\ldots,1}(1,2,3,\ldots,k),\qquad \gamma_2:=D^{-,+,+,\ldots,+}_{2,1,1,\ldots,1}(1,2,3,\ldots,k)
\]
which are related, as sets, via $\gamma_2=\varepsilon_1\cdot\gamma_1$.
Both cells $\gamma_1$ and $\gamma_2$ belong to the linear subspace
\[
      V_1=\big\{(x_1,\ldots,x_k)\in \R^{(d+1)\times k} : x_{1,1}=0\big\}.
\]
\item The equality $x_{1,r-1}=x_{1,r}$ for $2\leq r\leq k$ gives cells:
\[
         \gamma_{2r-1}:=D^{+,+,+,\ldots,+}_{1,\ldots,1,2,1,\ldots,1}(1,\ldots,r-1,r,r+1,\ldots,k),
\]
\[
         \gamma_{2r}:=D^{+,+,+,\ldots,+}_{1,\ldots,1,2,1,\ldots,1}(1,\ldots,r,r-1,r+1,\ldots,k),
\]
satisfying $\gamma_{2r}=\tau_{r-1,r}\cdot\gamma_{2r-1}$.
In these cells the index $2$ in the subscript $1,\ldots,1,2,1,\ldots,1$ appears at the position $r$.
These cells are contained in the linear subspace
\[
      V_r=\big\{(x_1,\ldots,x_k)\in \R^{(d+1)\times k} : x_{1,r-1}=x_{1,r}\big\}.
\]
\end{compactenum}

Let $e_{\theta}, e_{\gamma_1},\dots,e_{\gamma_{2k}}$ denote a generators in the cellular chain group corresponding to $\theta, \gamma_1,\dots,\gamma_{2k}$.
The boundary of the cell~$\theta$ is contained in the union of the linear subspaces $V_1,\ldots,V_k$.
Therefore we can orient the cells $\gamma_{2i-1},\gamma_{2i}$ consistently with the orientation of $V_i$, $1\leq i\leq k$, given in such a way that
\[
      \partial e_{\theta}=(e_{\gamma_1}+e_{\gamma_2})+(e_{\gamma_3}+e_{\gamma_4})+\cdots + (e_{\gamma_{2k-1}}+e_{\gamma_{2k}}).
\]
Consequently,
\begin{equation}
	\label{rel - 1}
	\partial e_{\theta}=(1+(-1)^d\varepsilon_1)\cdot e_{\gamma_1}+\sum_{i=2}^{k} (1+(-1)^d\tau_{i-1,i})\cdot e_{\gamma_{2i-1}}.
\end{equation}

\end{example}

\begin{example}\label{ex:cell_example}
Let $d \geq 1$, $k \geq 1$, and $1 \leq \ell \leq d-1$.
Consider the cell
\[
\theta:=D_{1+\ell,\ldots,1+\ell,1}^{+,\ldots,+,+}(1,2,\ldots,k)
\]
of of the $\Wk_k$-CW complex $X_{d,k}$, which is given by
\[
        0 <_{1+\ell} x_1 <_{1+\ell} x_2 <_{1+\ell} \dots <_{1+\ell} x_{k-1} <_1 x_k.
\]
More precisely, it is given by the inequalities
\[
        0=x_{1,1}=\dots=x_{1,k-1}<x_{1,k},\qquad 0=x_{r,1}=\dots=x_{r,k-1}
\]
for all $2\leq  r \leq \ell$, and
\[
        0 < x_{\ell+1,1} < \dots < x_{\ell+1,k-1}.
\]

\medskip
The cells of codimension one in the boundary of the cell $\theta$ are induced by addition of one of the following extra equalities:
\[
    x_{\ell+1,1}=0\,,\quad x_{\ell+1,1}=x_{\ell+1,2}\,,\quad \ldots,\quad x_{\ell+1,k-2}=x_{\ell+1,k-1}\,,\quad x_{1,k-1}=x_{1,k}.
\]
We have the following cells of codimension $1$ in the boundary of $\theta$:

\begin{compactenum}[\rm \quad (a)]

\item
The equality $x_{\ell+1,1}=0$ gives cells:
\begin{align*}
    \nu_1 &:=D^{+,+,+,\ldots,+,+}_{\ell+2,\ell+1,\ell+1,\ldots,\ell+1,1}(1,2,3,\ldots,k),	\\
    \nu_2 &:=D^{-,+,+,\ldots,+,+}_{\ell+2,\ell+1,\ell+1,\ldots,\ell+1,1}(1,2,3,\ldots,k),
\end{align*}
which on the level of sets are related by $\nu_2=\varepsilon_1\cdot\nu_1$.
Both cells $\gamma_1$ and $\gamma_2$ belong to the linear subspace
\[
    V_1=\big\{(x_1,\ldots,x_k)\in \R^{(d+1)\times k} : x_{1,1}=0,\ \dots,\ x_{\ell+1,1}=0\big\}.
\]

\item
The equality $x_{\ell+1,r-1}=x_{\ell+1,r}$ for $2\leq r\leq k-1$ induces cells:
\begin{align*}
    \nu_{2r-1}&:=D^{+,+,+,\ldots,+,+}_{\ell+1,\ldots,\ell+1,\ell+2,\ell+1,\ldots,\ell+1,1}(1,\ldots,r-1,r,r+1,\ldots,k),	\\
    \nu_{2r}&:=D^{+,+,+,\ldots,+,+}_{\ell+1,\ldots,\ell+1,\ell+2,\ell+1,\ldots,\ell+1,1}(1,\ldots,r,r-1,r+1,\ldots,k),
\end{align*}
satisfying $\nu_{2r}=\tau_{r-1,r}\cdot\nu_{2r-1}$.
In these cells the index $\ell+2$ in the subscript $\ell+1,\ldots,\ell+1,\ell+2,\ell+1,\ldots,\ell+1,1$ is at the position $r$.
These cells belong to the linear subspace
\[
    V_r=\big\{(x_1,\ldots,x_k)\in \R^{(d+1)\times k} : x_{1,r-1}=x_{1,r},\ \dots,\  x_{\ell+1,r-1}=x_{\ell+1,r}\big\}.
\]

\item
In the case $\ell = 1$ the last equality $(0=)x_{1,k-1}=x_{1,k}$ induces $2k$ cells for each $1 \leq r \leq k$ of the form
\begin{align*}
    \mu_{2r-1}&:=D^{+,+,\ldots,+,\ldots,+,+}_{2,2,2,\dots,2,2}(1,\dots,r-1,k,r,\dots,k-1),	\\
    \mu_{2r}&:=D^{+,+,\ldots,-,\ldots,+,+}_{2,2,2,\dots,2,2}(1,\dots,r-1,k,r,\dots,k-1),
\end{align*}
satisfying $\mu_{2r} = \varepsilon_r \mu_{2r-1}$.
The minus-sign is on the $r$-th position.

\item
In the case $\ell > 1$ the last equality $(0=)x_{1,k-1}=x_{1,k}$ induces $2$ cells of the form
\begin{align*}
    \mu_{2k-1}&:=D^{+,+,+,\dots,+,+}_{\ell+1,\ell+1,\dots,\ell+1,2}(1,2,3,\dots,k),\\
    \mu_{2k}&:=D^{+,+,+,\dots,+,-}_{\ell+1,\ell+1,\dots,\ell+1,2}(1,2,3,\dots,k),
    \end{align*}
satisfying $\mu_{2k} = \varepsilon_k \mu_{2k-1}$.
Either way these cells belong to the subspace
\[
    V_k=\big\{(x_1,\ldots,x_k)\in \R^{(d+1)\times k} : 0=x_{1,1}=\cdots=x_{1,k}\big\}.
\]
\end{compactenum}

\medskip
Let $e_{\theta}, e_{\nu_1},\ldots,e_{\nu_{2k-2}}$, $e_{\mu_1},\dots,e_{\mu_{2(k-1)}}$, $e_{\mu_{2k-1}},e_{\mu_{2k}}$
denote generators in the cellular chain group that correspond to the cells $\theta$, $\nu_1,\ldots,\nu_{2k-2}$, $\mu_1,\dots,\mu_{2(k-1)}$, $\mu_{2k-1},\mu_{2k}$, respectively.
The boundary of the cell~$\theta$ is a subset of the union of the linear subspaces $V_1,\ldots,V_k$.
Hence, we can orient the subspaces and the cells consistently in such a way that for $\ell > 1$  the following equality holds
      \[
      \partial e_{\theta}=(e_{\nu_1}+e_{\nu_2})+\cdots + (e_{\nu_{2k-3}}+e_{\nu_{2k-2}})+ (e_{\mu_{2k-1}} + e_{\mu_{2k}}),
      \]
while for $\ell = 1$ we get
      \[
      \partial e_{\theta}=(e_{\nu_1}+e_{\nu_2})+\cdots + (e_{\nu_{2k-3}}+e_{\nu_{2k-2}})+ (e_{\mu_{1}} + e_{\mu_{2}}) + \dots + (e_{\mu_{2k-1}} + e_{\mu_{2k}}).
      \]
Thus,
\begin{align}
    \label{rel - 2}
\partial e_{\theta}&=(1+(-1)^{d-1}\varepsilon_1)e_{\nu_1}+\sum_{i=2}^{k-1}(1+(-1)^{d-1}\tau_{i-1,i})e_{\nu_{2i-1}}+ \\
&\qquad\qquad\qquad\qquad\qquad\qquad\qquad\qquad\qquad\qquad\qquad \sum_{i=w}^{k} (1+ (-1)^{d} \varepsilon_{i})e_{\mu_{2i}}, \nonumber
\end{align}
where $w={\small \begin{cases}1, & \ell = 1\\ k, & \ell \neq 1\end{cases}}$.

\end{example}


\end{document}